\documentclass[11pt]{amsart}

\usepackage{amsmath,amsthm,verbatim,amssymb,amsfonts,amscd, graphicx,color, framed,enumerate,stackrel}
\usepackage[left=1in,right=1in,top=1in,bottom=1in]{geometry}
\usepackage[colorlinks]{hyperref}
\usepackage{tikz}
\usepackage[most]{tcolorbox}
\usepackage{chemfig, chemnum}
\newtcolorbox[blend into=figures]{myfigure}[2][]{float=htb, title={#2},#1}

\usepackage[normalem]{ulem}

\usepackage[foot]{amsaddr}

\def\ds{\displaystyle}
\def\R{\mathbb R}
\newcommand{\RR}{\mathbb{R}}
\def\Z{\mathbb Z}

\def\Im{\text{Im}}

\def\dim{\text{dim}}
\def\var{\text{Var}}
\def\P{\mathbb{P}}
\def\E{\mathbb{E}}

\newtheorem{theorem}{Theorem}[section]
\newtheorem{lemma}[theorem]{Lemma}

\newtheorem{proposition}[theorem]{Proposition}

\newtheorem{corollary}[theorem]{Corollary}

\theoremstyle{definition}

\newtheorem{definition}[theorem]{Definition}

\newtheorem{example}[theorem]{Example}
\AtBeginEnvironment{example}{%
	\pushQED{\qed}%
}
\AtEndEnvironment{example}{\popQED\endexample}

\newtheorem{remark}[theorem]{Remark}

\newtheoremstyle{example_contd}
{0.3cm}% Space above
{0.3cm}% Space below
{\upshape}% Body font
{}% Indent amount (empty = no indent, \parindent = para indent)
{\bfseries\scshape}% Thm head font
{.}% Punctuation after thm head
{0.5em}% Space after thm head (\newline = linebreak)
{ \thmname{#1} \thmnumber{ #2}\thmnote{#3} (continued)}% Thm head spec

\theoremstyle{example_contd}
\newtheorem*{example_contd}{\hspace*{-0.45em}Example}
\AtBeginEnvironment{example_contd}{%
	\pushQED{\qed}%
}
\AtEndEnvironment{example_contd}{\popQED\endexample}

\begin{document}

\title[Prevalence of multistationarity and absolute concentration robustness]{Prevalence of multistationarity and  \\ absolute concentration robustness in reaction networks}

\author{Badal Joshi}
\address[Badal Joshi]{California State University, San Marcos}

\author{Nidhi Kaihnsa}
\address[Nidhi Kaihnsa]{University of Copenhagen}

\author{Tung D. Nguyen}
\author{Anne Shiu}
\address[Tung D. Nguyen and Anne Shiu]{Texas A\&M University}
\email{bjoshi@csusm.edu,nidhi@math.ku.dk,daotung.nguyen@tamu.edu,annejls@math.tamu.edu}
 % 60 is probability

% SET THIS FOR ARXIV
\date{\today}
\maketitle

\vspace{-0.5cm}

\begin{abstract}

For reaction networks arising in systems biology, % and other applications, 
the capacity for two or more steady states, that is, multistationarity, is an important property that underlies biochemical switches. 
%Besides multistationarity, 
Another property receiving much attention recently is absolute concentration robustness (ACR), which means that some species concentration is the same at all positive steady states. 
In this work, we investigate the prevalence of each property while paying close attention to when 
the properties occur together. 
%reaction networks possess both properties.  
Specifically, we consider a stochastic block framework for generating random networks, and prove edge-probability thresholds at which – with high probability – multistationarity appears and ACR becomes rare. 
We also show that the small window in which both properties occur only appears in networks with many species.  
% Our findings indicate that reaction networks often exhibit ACR when they are sparse enough. 
% On the contrary, when reaction networks are dense enough, they are likely multistationary. Reaction networks exhibiting both properties are rare, and the window for the co-existence of both properties does not exist unless the number of species are very large.
Taken together, our results confirm that, in random reversible networks, ACR and multistationarity together, or even ACR on its own, is highly atypical. 
 Our proofs rely on two prior results, one pertaining to the prevalence of networks with deficiency zero, and the other “lifting” multistationarity from small networks to larger ones.

\vskip 0.1in

\end{abstract}

\section{Introduction}\label{sec:introduction}

In biochemical reaction networks, multistationarity is often a desirable phenomenon, as it is associated with biochemical switches, cellular signaling, 
% ANNE: Let's use the Oxford comma
and decision-making~\cite{tyson-albert}.  
A network is multistationary when there are two or more compatible positive steady states; ``compatible'' means that the steady states have the same conserved quantities such as total mass.  Which reaction networks are multistationary?  This question has a long history, 
and many results have been established (see the survey~\cite{mss-review}).
%The question of what reaction networks exhibit multistationarity has a long-standing history, and many results have been established (see the survey~\cite{mss-review}). 

Another significant property exhibited by some biochemical reaction networks is absolute concentration robustness (ACR), which refers to when a steady-state species concentration is maintained even when initial conditions are changed.  The concept of ACR was first popularized by Shinar and Feinberg in 2010~\cite{ACR} and has since attracted much interest both from the mathematical standpoint \cite{anderson2017finite, foundation_ACR, MST, joshi2022motifs} and in applications~\cite{briat2016antithetic, kim2020absolutely}.
%Morally, multistationarity and ACR are opposite behaviors. 
%\textcolor{blue}{[If both properties are present, the steady states cannot be in general position. 
%For a parametrized family of dynamical systems (such as mass action systems), this is unlikely to occur for a positive measure set of parameters.
%An exception to this is when the network itself has some special symmetry properties. In this work, we establish that if such symmetry properties do exist, they are rare and vanish asymptotically.] Probably need to move this part or integrate it into the other parts of the intro.}
%However, checking whether a given network has either of these properties is generally difficult~\cite{mss-review,MST}. 

While each of these two properties has been studied in isolation, their relationship is not well understood. 
Nonetheless, multistationarity and ACR can be viewed as opposite behaviors, as multiple steady states cannot be in general position if ACR is present. % It is not surprising that 
Indeed, known examples of networks having ACR (for instance, those in \cite{ACR}) typically are non-multistationary. 

Accordingly, the driving motivation of this article is to explore the relationship between multistationarity and ACR and to investigate the prevalence of networks with either property. 
However, it is generally challenging to assess multistationarity and ACR~\cite{mss-review,MST}. Therefore, following the approach of Anderson and Nguyen~\cite{prevalence_block, prevalence}, we instead investigate the {\em prevalence} of these properties in randomly generated reaction networks.  Specifically, we prove asymptotic results on the probability that such a network has either property, as the number of species goes to infinity. 

A summary of our results appears in Table~\ref{tab:summary-prevalence}, which pertains to when reaction networks are randomly generated by a certain {\em stochastic block model} in which the expected numbers of reactions of each ``type" are roughly of the same magnitude.  Here, the type refers to which forms of complexes -- $0, X_i, 2X_i, X_i+X_j$ -- appear in the reaction (notice that we restrict our attention to  \textit{at-most-bimolecular} networks, which encompass most reaction networks arising in biochemistry).  We prove edge-probability thresholds for the resulting reaction networks to be nondegenerately multistationary or to preclude ACR (Theorem \ref{thm:maintheorem}), and a restatement of this result in terms of the expected number of edges (that is, reactions) is in Table~\ref{tab:summary-prevalence}.

% ------------------
% SUMMARY TABLE
% ------------------
\begin{table}[ht]
\begin{center}
\begin{tabular}{l c c}
\hline
Expected number of reactions & Multistationary? & ACR? \\
\hline
% ------------------
Asymptotically greater than $1$, but less than $n^{2/3}$ & No & Yes \\
% ------------------
Asymptotically greater than $n$, but  less than $\frac{2}{17}n(\log(n)-c(n))$, 
	& Yes & Yes \\
\hspace{.1in}  for some $c(n)\to\infty$ \\
% ------------------
Greater than $n(\log(n)+c(n) )$,   for some $c(n)\to\infty$
 & Yes & No \\
% ------------------
\hline
\end{tabular}
\end{center}
\caption{For random reaction networks, with $n$ species, %obtained under 
generated by 
a certain stochastic block model, this table lists ranges for the expected number of reactions, and whether -- with high probability -- such a network is multistationary or has ACR.  For further details, see Theorem \ref{thm:maintheorem} and Remark \ref{remark:thresholds}.}
\label{tab:summary-prevalence}
\end{table}%
% ------------------

\vspace{-0.4cm}

We see from Table~\ref{tab:summary-prevalence} that the window for having both multistationarity and ACR is relatively small:
when the expected number of edges is asymptotically between $n$ and $\frac{2}{17}n\log(n)$.  In fact, in this window, 
which only appears when there are thousands of species (Remark~\ref{remark:window}), 
the ACR species either appears by itself without interacting with other species (in essence, the network decouples) or it appears as a catalyst (Remark \ref{remark:decouple}). We therefore expect that reaction networks with both multistationarity and ACR, in which the ACR species interacts nontrivially with other species, are rare and may require special structures or constructions.

Our proofs rely crucially on two prior results.  
The first, due to Anderson and Nguyen, pertains to the prevalence of certain networks that are known to preclude multistationarity, namely,
networks with deficiency zero (specifically, the result asserts that ``sparse'' reaction networks are likely to have deficiency zero)~\cite{prevalence_block, prevalence}.
The second result concerns ``lifting'' multistationarity from small networks to larger 
ones~\cite{splitting-banaji, lifting_mss, atoms_multistationarity}, 
which we use to show the high probability of multistationarity and the absence of ACR in ``dense'' reaction networks.

This article is structured as follows: Section~\ref{sec:preliminaries} introduces reaction networks, multistationarity, and ACR. 
Section~\ref{sec:prevalence} contains our results on the prevalence of multistationarity and ACR in random reaction networks via 
edge-probability 
thresholds. In Section \ref{sec:stochastic_block}, we introduce the type-homogeneous stochastic block model, which we use to generate random reaction networks, and compute explicitly the thresholds for multistationarity and ACR.
We end with a discussion in Section~\ref{sec:discussion}.

\section{Background}\label{sec:preliminaries}
%{\color{teal} What's in brown is identical (?) to that of the other paper.  Need to edit!}

Here, we recall the basic setup and definitions involving reaction networks (Section~\ref{sec:rxn-network}), 
the dynamical systems they generate (Section~\ref{sec:mass-action}), and absolute concentration robustness (Section~\ref{sec:ACR-system}). 
(A more detailed exposition can be found in~\cite{feinberg2019foundations}.) 
%For a more detailed exposition, interested readers are referred to \cite{feinberg2019foundations}. 
We also discuss how multistationarity and ACR are %subjected to 
affected when 
adding new reactions to an existing network (Section \ref{sec:monotone}).

\subsection{Reaction networks} \label{sec:rxn-network}

A {\em reaction network} $G$ is a directed graph in which the vertices are non-negative linear combinations of {\em species} $X_1,X_2, \ldots, X_n$.  
As is standard in reaction network theory, we refer to each vertex as a {\em complex}, and we denote the $i$-th complex by 
$y_i= 
	y_{i1}X_1 + y_{i2} X_2 + \dots + y_{in} X_n $ 
or by $y_i=( y_{i1}, y_{i2}, \dots, y_{in})$ 	
%$y_i=\sum_{j=1}^{n}y_{ij}X_j$ 
(where $y_{ij}\in\Z_{\geq 0}$).

Edges of $G$ represent the possible changes in the abundances of the species, and are referred to as {\em reactions}. It is standard to represent a reaction $(y_i, y'_i)$ by the notation $y_i \to y'_i$.  
%{\color{red} [Here we write a reaction as $y_i \to y_j$ and later as $y_i \to y_{i'}$.  Do we want to be consistent?]} {\color{blue} [Fixed.]}
In such a reaction,  $y_i$ is the \textit{reactant complex}, and $y'_i$ is the \textit{product complex}. 
A species $X_k$ is a {\em catalyst-only species} of a reaction $(y_i, y'_i)$ if the stoichiometric coefficient of $X_k$ is the same in the product and reactant (that is, $y_{ik} = y'_{ik}$).
%If, for some species $X_k$, along a reaction $(y_i, y_j)$ the units, $y_{ik}$ and $y_{jk}$, of $X_k$ involved are equal, we call $X_k$ {\em catalyst-only species} in that reaction.
Finally, 
a reaction network $G'$ is a {\em subnetwork} of a network $G$ if the sets of species, complexes, and reactions of $G'$ are subsets of the respective sets of $G$.
%set of species, complexes and reactions of $G$ respectively.

In examples, it is often convenient to write species as $A,B,C,\dots$ rather than $X_1,X_2,X_3,\dots$. 
Additionally, we typically depict a reaction network by its
%a reaction network can be {\color{teal} fully described} by its 
set of reactions,
in which case the sets of 
 %since the sets of 
 species and complexes are implied.
 % in such a case. 

\begin{example} \label{ex:generalized-degenerate-network}
%For instance, 
The reaction network $\{ A+B\to 2B ,~ B \to A\}$ has 2 species (namely, $A$ and $B$), 4~complexes ($A+B,~2B,~B,~A$), and 2 reactions.
\end{example}

A reaction network is \textit{reversible} if every edge in the graph is bidirected.  

\begin{example} \label{ex:reversible-mss-network}
	The reaction network $\{ 
	A \leftrightarrows B+C, ~
	0 \leftrightarrows A,~
	0 \leftrightarrows B,~
	C \leftrightarrows 2C
	\}$
	is reversible.
\end{example}

This article focuses on \textit{at-most-bimolecular} reaction networks (or, for short, {\em bimolecular}), which means that every complex $y_i$ of the network satisfies 
$y_{i1}+y_{i2}+ \dots + y_{in} \leq 2$ (where $n$ is the number of species). 
%$\sum_j^n y_{ij}\leq 2$. 
Equivalently, each complex has the form $0$, $X_i$, $X_i+X_j$, or $2X_i$ (where $X_i$ and $X_j$ are species). The reaction networks in  Examples~\ref{ex:generalized-degenerate-network}--\ref{ex:reversible-mss-network} are bimolecular.

\subsection{Dynamical system arising from a network} \label{sec:mass-action}
Under the assumption of mass-action kinetics, each reaction network $G$ defines a  parametrized family of systems of ordinary differential equations (ODEs), as follows. 
Let $r$ denote the number of reactions of $G$. We write the $i$-th reaction as $y_i \to y_i'$ % (equivalently, $\sum_{j=1}^ny_{ij}X_j\longrightarrow \sum_{j=1}^ny_{ij}'X_j$).  
and assign a positive 
\textit{rate constant} 
 $\kappa_{i}\in \RR_{> 0}$ to the corresponding reaction.

The \textit{mass-action system}, denoted by $(G,\kappa)$, where $\kappa=(\kappa_1,\kappa_2, \dots, \kappa_r)$, 
is the dynamical system arising from the following ODEs:
%is the following dynamical system  that arises from assigning mass-action kinetics to each reaction: 

\begin{equation}\label{eq:mass_action_ODE}
		\frac{dx}{dt} ~=~ \sum_{i=1}^r  \kappa_{i} x^{y_i} (y_i'-y_i) ~=:~ f_{\kappa}(x)~,
\end{equation}
where  $x_i(t)$ denotes the concentration of the $i$-th species at time $t$ and $x^{y_i} := \prod_{j=1}^n x_j^{y_{ij}}$.
The right-hand side of the ODEs~\eqref{eq:mass_action_ODE} consists of polynomials 
$f_{\kappa,i}$, 
for $i=1,2,\dots,n$ (where $n$ is the number of species). For simplicity, we often write $f_i$ instead of $f_{\kappa,i}$.

The \textit{stoichiometric subspace} of $G$, which we denote by $S$, is the linear subspace of $\mathbb{R}^n$ spanned by all reaction vectors $y_i' - y_i$ (for $i=1,2,\dots, r$).  
When $\dim(S)=n$, we say that $G$ is {\em full dimensional}.  
Observe that the vector field of the mass-action ODEs~\eqref{eq:mass_action_ODE} lies in $S$ (more precisely, the vector of ODE right-hand sides is always in $S$).  Hence,  
% Note that the stoichioimetric subspace is a property of $G$ only 
%(with no dependence of $\kappa$).  
a forward-time solution $\{x(t) \mid t \ge 0\}$ of~\eqref{eq:mass_action_ODE}, with initial condition $x(0)  \in \R_{>0}^n$, remains in the following \textit{stoichiometric compatibility class}~\cite{feinberg2019foundations}: 
% Anne added a citation b/c the non-negative part is nontrivial
$$P_{x(0)} ~:=~ (x(0)+S)\cap \RR_{\geq 0}^n~.$$

\begin{example_contd}[\ref{ex:generalized-degenerate-network}] \label{ex:generalized-degenerate-network-ode}
	%For instance, 
The network $\{ A+B \overset{\kappa_1}{\rightarrow} 2B,~ B \overset{\kappa_2}{\rightarrow} A\}$ 
generates the following mass-action ODEs~\eqref{eq:mass_action_ODE}:

	\begin{align*}
	&\frac{dx_1}{dt} ~=~ -\kappa_1x_1x_2+\kappa_2x_2\\
	&\frac{dx_2}{dt} ~=~ \kappa_1x_1x_2 - \kappa_2x_2.
	\end{align*}
Moreover, it has a one-dimensional stoichiometric subspace (spanned by the vector $(1,0)^\top$). 
\end{example_contd}

A {\em steady state} of a mass-action system is a nonnegative concentration vector $x^*\in \R_{\geq 0}^n$ at which the right-hand side of the ODEs \eqref{eq:mass_action_ODE} vanishes: $f_{\kappa}(x^*)=0$. 
Our primary interest in this work is in {\em positive} steady states $x^*\in \R_{> 0}^n$. 
%{\color{violet} \sout{
%The set of all positive steady states of a mass-action system 
%can have positive dimension in $\RR^n$, 
%but with the exception of degenerate cases it  intersects each stoichiometric compatibility class in finitely many points. } [Anne proposes to delete this sentence -- or we need to clarify what we mean by degenerate here.]} {\color{blue} [I am ok with deleting it.]}
Also, a steady state $x^*$ is \textit{nondegenerate} if $\Im(df_{\kappa}(x^*)|_S)=S$, where $df_{\kappa}(x^*)$ is the Jacobian matrix of $f_\kappa$ evaluated at $x^*$, and $S$ is the stoichiometric subspace.

We consider multistationarity at two levels: systems and networks.  
A mass-action system $(G,\kappa)$ is \textit{multistationary} (respectively, {\em nondegenerately multistationary}) 
 if there exists some stoichiometric compatibility class having more than one positive steady state (respectively, nondegenerate positive steady state). 
A reaction network $G$ is \textit{multistationary} if there exist positive rate constants $\kappa$ such that $(G,\kappa)$ is multistationary. 
% (respectively, {\em nondegenerately multistationary}).  
Similarly, a network $G$ can be {\em nondegenerately multistationary}.

\begin{example}[Multistationary system] \label{eg:multimotifrates} 
%\begin{example}[A system with 3 species, 3 positive steady states] \label{eg:multimotifrates} 
Consider the mass-action~system generated by: 
	\begin{align} \label{eq:motif-with-rate-constants}
	& \left\{
	A \stackrel[1]{1}{\leftrightarrows} B+C, \quad 0 \stackrel[6]{1}{\leftrightarrows} A, \quad 0 \stackrel[27]{1}{\leftrightarrows} B, \quad C \stackrel[8]{1}{\leftrightarrows} 2C
		\right\}~.
	\end{align}
It is straightforward to check (or compute) that there are exactly three positive steady states, namely, 
	$(13,20,1),(18,15,2),(21,12,3)$, 
and all three are nondegenerate.
\end{example}

\subsection{Deficiency and absolute concentration robustness} \label{sec:ACR-system}

The {\em deficiency} of a reaction network $G$ is 
$\delta = v- \ell - \dim(S)$, where $v$ is the number of vertices (or complexes) of $G$, $\ell$ is the number of connected components of $G$, and $S$ is the stoichiometric subspace. This invariant is %important for various classifications of dynamical properties of reaction networks based on their graphical structure and 
central to many classical results pertaining to mass-action systems~\eqref{eq:mass_action_ODE}~\cite{AC:non-mass,ACK:product,AN:non-mass,F1,H,H-J1}, including a structural criterion for absolute concentration robustness (ACR)~\cite{ACR}, the topic we turn to next. 

ACR, like multistationarity, is analyzed at the level of systems and also networks.

\begin{definition}[ACR] \label{def:acr}
Let $X_i$ be a species of a reaction network $G$ with $r$ reactions.
\begin{enumerate}
\item For a fixed vector of positive rate constants $\kappa \in \mathbb{R}^r_{>0}$, 
	the mass-action system $(G,\kappa)$ has {\em absolute concentration robustness} (ACR) in $X_i$ if $(G,\kappa)$ has a positive steady state and in every positive steady state $x \in \RR_{> 0}^n$ of the system, the value of $x_i$ (the concentration of $X_i$) is the same. This value of $x_i$ is the \textit{ACR-value} of $X_i$. 

%\item The reaction network $G$ \textit{has the capacity for ACR in $X_i$} (or simply \textit{has ACR in $X_i$}) if there is a $\kappa \in \R^r_{> 0}$ such that $(G,\kappa)$ has ACR in $X_i$. 

\item The reaction network $G$ 
has \textit{unconditional ACR}  in species $X_i$ if the mass-action system $(G,\kappa)$ has ACR in $X_i$ for all $\kappa \in \mathbb{R}^r_{>0}$.
\end{enumerate}
\end{definition}

When $G$ has unconditional ACR in $X_i$, the property of ACR in $X_i$ holds across all rate constants, but the ACR-value can (and typically does) change with rate constants, as in the next example.

% Shinar-Feinberg
\begin{example_contd}[\ref{ex:generalized-degenerate-network}]  We return to the following network: $\{
A +B \xrightarrow{\kappa_1} 2B, ~ B \xrightarrow{\kappa_2}A\}.$
This network is a classical example of a network with ACR \cite{ACR}.  Indeed, at all positive steady states, the concentration of species $A$ is $\kappa_2/\kappa_1$, and hence the network has unconditional ACR in $A$.
%The mass-action system has exactly one positive steady state $(\kappa_2/\kappa_1, T - \kappa_2/\kappa_1)$ IN EACH COMPATIABILTY CLASS, where $T$ is the total concentration of $A$ and $B$. Therefore, the network has unconditional ACR in $A$.
\end{example_contd}

%{\color{red} The ``Example, continued'' environment is nice, but can we add some vertical space before/after them, to match the other examples?}

%\blue{added the space. Feel free to change to more or less where the new theorem style is defined... it is now marked ``space above"}

%{\color{purple} We could put a triangle (instead of/ in addition to space) after Examples and Examples (continued). I already defined a macro for Examples. Re-use for Examples (continued) or remove.}

%{\color{brown} Thank you!  The only thing is: I don't like how we have Example 1, 2, etc. -- can we label like before, to match the section number?}

%{\color{blue} Also why is the continued example indented?}
%One simple way to do Example (continued) is the following (see \LaTeX code)
%\begin{example}[{\bf continued}]
%~
%\end{example}]

%\nidhinote{Indentation, Example numbers wrt Theorem numbers and triangles at the end of continued environment are all fixed now! Please check and if satisfied remove all comments including this.}. THANK YOU!

\begin{remark}[ACR and reversible networks] \label{rem:reversible}
In Definition~\ref{def:acr}, ACR requires the existence of a positive steady state.  This requirement is not included in some definitions of ACR in the literature.  However, in this work, we focus on the reversible networks, which %.  This simplifying assumption 
guarantees the existence of positive steady states (this result is due to Deng {\em et al.}~\cite{Deng} and Boros~\cite{boros2019existence}). Hence, our results are valid with or without the requirement of positive steady states.
\end{remark}

In the literature, reaction networks with ACR are typically not multistationary.
Nonetheless,
a network %$G$ 
can both 
have ACR (in some species) and be multistationary.
%can have ACR in a species $X$ as well as multistationarity. 
A simple example can be constructed by 
joining two networks, with disjoint species sets, where one network has ACR and the other is multistationary.
%taking any reaction network (with species set $\S_Y = \{Y_1, \ldots, Y_m\}$) 
%that is multistationary and then augmenting by some reaction network that has ACR in species $X_1$ and has species set $\S_X = \{X_1, \ldots, X_n\}$ disjoint from $\S_Y$ (that is, $\S_X \cap \S_Y = \varnothing$). 
%
A %slightly 
less trivial example can be generated by having the ACR species participate as an enzyme
 --  more precisely, 
as a catalyst-only species -- 
 in the multistationary network. 
%i.e., 
%$X_1$ has the same stoichiometric coefficient in both the reactant and the product complex of some reaction
%it appears as a catalyst-only species in the first network. 
We illustrate this in the following example.

\begin{example}[A network with multistationarity and ACR] 
%\begin{example}[A network with 2 species, 3 positive steady states, ACR] 
Consider the following network, 
in which $A$ is a catalyst-only species in the first two reactions:
 %, which we call $G$:
\begin{align*}
\left\{
	A \stackrel[\kappa_1]{\kappa_2}{\leftrightarrows} A + B, 
	\quad 
	2B \stackrel[\kappa_3]{\kappa_4}{\leftrightarrows} 3B, 
	\quad A 
	\stackrel[\kappa_5]{\kappa_6}{\leftrightarrows} 2A 
\right\}~.
\end{align*}  
This network, which we call 
$G$, generates the following mass-action ODEs~\eqref{eq:mass_action_ODE}:
\begin{align*}
&\frac{dx_1}{dt} ~=~ \kappa_5x_1 - \kappa_6x_1^2\\
&\frac{dx_2}{dt}~ =~ \kappa_1x_1-\kappa_2x_1x_2+\kappa_3x_2^2-\kappa_4x_2^3.
\end{align*}
One can check directly that $G$ has unconditional ACR in species $A$ with ACR-value $\kappa_5/\kappa_6$.
Moreover, for reaction rates $(\kappa_1, \kappa_2, \ldots,\kappa_6)=(\frac{1}{512},\frac{1}{16},1,1,2,1)$, we obtain exactly $3$ positive steady states, with the following approximate values: 
$( 2, 0.050987),$ $( 2, 0.0890928)$, and $ ( 2, 
	 0.85992)$.
%\[
%( 2, 0.050987), ( 2, 0.0890928), ( 2, 
%	 0.85992)~.
%	 \]
\end{example}
% roots of x^​3-​x^​2+​0.125*​x-​1/​256	 
%\color{blue}{ \sout{
%Furthermore, the steady state values of $x_2$ are the roots of a cubic with 3 sign changes. Thus $(G,\kappa)$ is multistationary only for a proper subset of the set of all vectors of positive rate constants. This subset has positive -- in fact, infinite -- Lebesgue measure.}}
%{\color{red} 
%I wonder if this description is too detailed for this paper.  Can we just assert that it is easy to check that $G$ can have $3$ positive steady states?
%} \nidhinote{The extra sentence in orange should be enough. please check.}

It is not straightforward to find non-trivial examples of reaction networks with multistationarity and unconditional ACR where the network cannot be decomposed into individual pieces, each with only one of the two properties. 
While such networks do exist, this is a topic of study unto its own. We will report on several families of such networks and their operating principles in future work. 
%{\cre (rearranged above and below.)}
%There do exist far more meaningful and interesting networks where ACR and multistationarity coexist in a nontrivial manner. Significantly, the networks do not decompose into individual pieces with each subnetwork having only one of the two properties. 
% Such families of networks will be reported and studied in depth in future work. 

In this work, we are interested in asymptotic results (as the size of the network grows) on the prevalence of multistationarity and ACR. 
An important tool we use for proving thresholds for these properties (or the lack thereof) 
is the network in the following example.
% in Sections \ref{sec:prevalence} and~\ref{sec:stochastic_block}.

% EXAMPLE CONTINUED
\begin{example_contd}[\ref{eg:multimotifrates} ]
Consider again the following mass-action system:
	\begin{align*} %\label{eq:motif-with-rate-constants}
	& \left\{
	A \stackrel[1]{1}{\leftrightarrows} B+C, \quad 0 \stackrel[6]{1}{\leftrightarrows} A, \quad 0 \stackrel[27]{1}{\leftrightarrows} B, \quad C \stackrel[8]{1}{\leftrightarrows} 2C
		\right\}~,
	\end{align*}
which we saw has three positive steady states: $(13,20,1),(18,15,2),(21,12,3)$. 
By inspection, this system has no ACR (in any species) and hence the network does not have unconditional ACR. 
\end{example_contd}

We end this subsection by recalling what is known about multistationarity and ACR for networks with deficiency 0.  In the following result, part (1) follows from the deficiency-zero theorem~\cite{H} and part (2) is immediate from a recent result of Joshi and Craciun~\cite[Theorem 6.1]{foundation_ACR}.

\begin{lemma} \label{lem:def-0}
	If $G$ is a reaction network that has deficiency 0, then:
	\begin{enumerate}
		\item $G$ is not multistationary, and 
		\item if $G$ contains an inflow or outflow reaction (that is, a reaction of the form $0 \to X_i$ or $0 \leftarrow X_i$, for some species $X_i$), then $G$ has unconditional ACR (in some species).
	\end{enumerate}
\end{lemma}

\subsection{Monotonicity of multistationarity and non-ACR with respect to adding reactions}\label{sec:monotone}
This subsection pertains to how multistationarity and ACR are affected as we add new reactions to a reaction network. 
The following proposition essentially follows from recent results on ``lifting'' multistationarity from smaller networks to larger 
ones~\cite{splitting-banaji, lifting_mss, atoms_multistationarity}.
% LEMMA
\begin{lemma}[Lifting multistationarity or non-ACR]\label{lem:monotone}
Let $G$ be a full-dimensional network, and 
let $G'$ be a network obtained by adding to $G$ a reaction that involves no new species (new complexes are allowed). 
%\red{Here, by ACR we mean unconditional right? I ask because the argument only generates a reaction rate for which system doesn't have ACR, it doesnt say anything for other rates.}
\begin{itemize}
\item[(i)] If $G$ is nondegenerately multistationary, then so is $G'$.
\item[(ii)] If  {there exists a vector of positive rate constants $\kappa^*$ such that $(G,\kappa^*)$ is nondegenerately multistationary and also does not have ACR (in any species),}
%\sout{ $G$ admits nondegenerate multistationarity and does not have unconditional ACR in any species, }
then the network $G'$ does not have unconditional ACR in any species. 
\end{itemize}
\end{lemma}

\begin{proof}
Part $(i)$ follows directly from \cite[Theorem 3.1]{atoms_multistationarity}. 

For part $(ii)$, suppose that there exists $\kappa^*$ such that $(G,\kappa^*)$ does not have ACR (in any species) and also has nondegenerate, positive steady states $q_1,q_2,\dots,q_m$, where $m \ge 2$.
We denote each steady state by $q_i=(q_{i,1},\ldots, q_{i,n})\in \RR^n_{\ge 0}$ (for $i=1,\ldots, m$), where $n$ is the number of species of $G$.

Let $\epsilon$ denote the rate constant of the reaction added to $G$ to obtain $G'$.
From the proof of \cite[Theorem 3.1]{atoms_multistationarity}, there exists $\epsilon_0>0$ such that if $0 < \epsilon < \epsilon_0$, then 
	$(G', (\kappa; \epsilon))$ 
%if $\epsilon$ is the rate constant of the reaction added to $G$ to obtain $G'$, then $G'$
has nondegenerate, positive steady states $q_{1}(\epsilon),q_{2}(\epsilon),\ldots,q_{m}(\epsilon)$ such that $\lim_{\epsilon\to 0^+}q_{i}(\epsilon) = q_i$ (for all $i$). 
%% \cite[Chapter 9]{Rahman:Schmeisser:AnalyticTheoryOfPolynomials}).  
%    \sout{Hence, for $\epsilon$ sufficiently small, the $q_{i}(\epsilon)$'s are distinct, and so $(G', (\kappa^*; \epsilon))$
%    has $m$ nondegenerate positive steady states. }
%    [Anne proposes to delete this, since we only need to prove non-ACR in part (ii) of the proof.]}

Next, consider some species $X_\ell$ (so, $1 \leq \ell \leq n$).  As
$(G,\kappa^*)$ does not have ACR in $X_\ell$, 
there exist 
steady states $q_i$ and $q_j$ at which the corresponding concentrations of $X_\ell$ differ (that is, $q_{i,\ell} \neq q_{j,\ell}$). 
So, as
$\lim_{\epsilon\to 0^+}q_{i}(\epsilon) = q_i$ and 
$\lim_{\epsilon\to 0^+}q_{j}(\epsilon) = q_j$,
%Now, by continuity of polynomials, 
there exists $\epsilon_\ell>0$ (with $\epsilon_\ell < \epsilon_0$) such that if $0<\epsilon<\epsilon_\ell$, then $|q_{i,\ell}(\epsilon)-q_{j,\ell}(\epsilon)|>0$
and hence $(G', (\kappa^*; \epsilon))$ does not have ACR in $X_\ell$.  

Finally, we pick $\epsilon$ such that $0<\epsilon < \min_\ell \epsilon_\ell$. 
%It is straightforward to check that for this choice of $\epsilon$, 
By construction, the system $(G', (\kappa^*; \epsilon))$ does not have ACR in any species. Hence, $G'$ does not have unconditional ACR.
\end{proof}

\begin{remark}\label{rem:monotone}
 Lemma \ref{lem:monotone} implies that, given full dimensionality, nondegenerate multistationarity is a \textit{monotone increasing} property (with respect to adding new edges/reactions).  
% {\color{red} \sout{
%On the other hand, 
%the proposition states that, {\color{violet} under some hypotheses},
%ACR is a \textit{monotone decreasing} property (equivalently, absence of ACR is a monotone increasing property).  } [This doesn't seem very precise, given that we have updated part (ii) of the proposition. Delete?]} {\color{blue} [Yes I am ok with deleting it.]}
\end{remark}

\section{Multistationarity and ACR in random reaction networks}\label{sec:prevalence}
In this section, we follow the approach in \cite{prevalence_block, prevalence} in which reaction networks are generated using a random-graph framework (Section~\ref{sec:random-setup}). 
In 
Section~\ref{sec:randomgeneral}, 
we prove the existence of thresholds for the presence or absence of 
nondegenerate multistationarity and unconditional ACR.
These thresholds are with respect to increasing {\em graph density}, that is, the fraction of reactions present  -- among all possible reactions.  
%{\color{violet} \sout{
%In Section~\ref{sec:stochastic_block}, we explicitly compute these thresholds for a specific random-graph model: a stochastic block model in which the edge probablities allow all ``types" of reactions and complexes to be equally represented. }}

In this section, we use the following standard notation.
For sequences of numbers $\{a_n\}$ and $\{b_n\}$, we write $a_n \ll b_n$ (or $b_n\gg a_n$) if
\[
\lim_{n\to \infty} \frac{a_n}{b_n} = 0~;
\]
and we write $a_n\sim b_n$ if 
\[
\lim_{n\to \infty} \frac{a_n}{b_n} = c~,
\]
for some positive constant $c$.
Also, a sequence of events $\{A_n\}$ occurs \textit{with high probability (w.h.p.)} if $\lim_{n\to\infty}\P(A_n)=1$.

\subsection{Random reaction networks}\label{sec:random-setup}

Consider the class of bimolecular reaction networks on  $n$ species $X_1,X_2,\dots,X_n$.
The set of all possible complexes is then
\begin{align}\label{eqn:vertices}
V_n ~=~ 
\{0 \}
~\cup~
\{X_i \mid 1 \le i \le n \}  ~\cup~ 
\{2X_i \mid 1 \le i \le n \}
~\cup~
\{X_i+X_j \mid1 \le i ,j \le n ~,~ i \neq j\}. % {\rm ~and~} 
\end{align}
The cardinality %, $N_n$, 
of $V_n$ is therefore given by
\[
%N_n 
|V_n|
	~=~ 1+n+n+ {n \choose 2} 
	%\frac{n(n-1)}{2}
	~=~ \frac{n^2+3n+2}{2}~.
\]

\begin{definition}[Edge probabilities] \label{def:edgeprob} 
	Let $n$ be a positive integer.
	\begin{enumerate}
	\item 
	Consider two distinct vertices $u,v \in V_n$.  An {\em edge probability function} for the unordered pair $e=(u,v)$ is a non-decreasing function, $\phi_{e}(p_n)$, in a single parameter $p_n \in[0,1]$. 
	\item 
	A {\em choice of edge probabilities} is a collection of edge probability functions, 
$\phi_{e}(p_n)$, 
one for each unordered pair $e=(u,v)$ of vertices in $V_n$.
	\end{enumerate}
\end{definition}

% The probability of an edge between given two vertices is encoded by an \textit{edge function} $\phi(u,v)$. 		
%In the remainder of the text we consider specific class of edge functions where the probabilities between any two vertices are dependent on exactly one parameter $p_n\in [0,1].$ We restrict our focus to functions such that, for fixed $u,v \in V_n$, the probability that there is an edge between any them is a monotonically increasing function, $\phi_{\{u,v\}}(p_n)$, with respect to the parameter $p_n \in[0,1]$. For convenience, we call such functions {\em edge probabilities}.

%To generate random networks with these species, we first g
\begin{definition}[Random graph $G(V_n,p_n)$] \label{def:rdm-graph-general} 
Fix a positive integer $n$, some $p_n \in [0,1]$, and a choice of edge probabilities $\{ \phi_{e}(p_n) \}$. 
We generate random (undirected) graphs, which we denote by $G(V_n,p_n)$, as follows: 
\begin{itemize}
	\item the vertex set is $V_n$, and 
	\item the probability that there is an edge between two vertices $u,v \in V_n$ is given by the corresponding edge probability function (where $e=(u,v)$):
		\[
		\P(e \text{ is an edge of } G(V_n,p_n))
		~=~
		\phi_{e}(p_n).
		\]
\end{itemize}
\end{definition}
%For convenience, we call such functions {\em edge probabilities}. For a given $n$, 
%We denote such a random graph by $G(V_n,p_n).$ 

%{\color{violet}
%Nowhere in the entire paper afterwards do we say anything about the function $\phi_{\{u,v\}}$. In fact we use vague choice of words like `` specific choice of edge probabilities" where in fact I think we mean $\phi_{\{u,v\}}$. Either remove this function here or point this out later... for eg in the introduction of section 4. It is worth saying a specific choice of $\phi$ gives block diagonal and we work with that choice. Our language is unclear at bunch of such points.}. {\color{brown} I see the concern!  Each function is an edge probability (function), right?  So, that's why we say a ``choice of edge probabilities''.  
%(Maybe we can add here, ``A {\em choice of edge probabilities} is a collection of edge probabilities $\phi_{\{u,v\}}(p_n)$, one for each pair of vertices $u,v$''.)
%I think the function notation is important here, to make it clear that the function can depend on which 2 vertices are involved.  To discuss! (Anne)}

\begin{definition}[Random network $G_n$] \label{def:rdm-network-general} 
Each random graph $G(V_n, p_n)$ (generated by some choice of edge probabilities) defines a random reaction network, which we denote by $G_n$, consisting of reversible reactions, as follows:
\begin{itemize}
\item The set of species of $G_n$ is $\{X_1,X_2,\dots ,X_n\}$.
%\item {\color{brown} \sout{The complexes of $G_n$ are the non-isolated vertices of $G(V_n, p_n)$.}}
\item The (reversible) reactions of $G_n$ correspond to the edges of $G(V_n, p_n)$.
\end{itemize}
\end{definition}
Recall from Section \ref{sec:rxn-network} that $G_n$ is full-dimensional if its dimension is $n$.

%{\color{violet} We didn't define degree for graphs anywhere. [NK]}

\begin{remark} \label{rem:missing-species}
It is possible that some species of $G_n$ appears in no complexes, 
%may not appear in any non-isolated vertices, 
especially when $G(V_n,p_n)$ is sparse (e.g., the network shown later in Figure~\ref{fig1}).
% Example \ref{example:sparse}). 
Such $G_n$ are not full-dimensional. % and these species' concentrations stay constant.
\end{remark}

\subsection{Thresholds for multistationarity and ACR}\label{sec:randomgeneral}

In this subsection, we 
show that for the random reaction networks defined in the prior subsection, 
there exist thresholds for the presence or absence of 
nondegenerate multistationarity and unconditional ACR (Theorem~\ref{theorem:threshold}).  Subsequently, we discuss the challenges of computing such thresholds, and then describe a strategy for proving upper bounds on the thresholds (see Corollary~\ref{cor:thresholds-equal}).

The following definition is useful in the proof of Theorem~\ref{theorem:threshold} and also later in Proposition~\ref{prop:MM+LC}.

\begin{definition}[$S^*_n$]\label{def:S*}
For $n\in\Z_{\geq 1}$, 
let $S^*_n$ denote the set of all full-dimensional bimolecular networks $G$ with exactly $n$ species % {\color{red} [Tung: I think it should be exactly n here.]}
for which there exists a vector of rate constants $\kappa$ such that $(G,\kappa)$ is nondegenerately multistationary and also does not have ACR in any species. 
\end{definition}

\begin{remark}[$S^*_n$ is nonempty for $n\geq 2$] 
The set $S^*_1$ is empty~\cite{joshi2017small}, but for all $n\geq 2$, $S_n^*$ is nonempty.  This is shown for $n\geq 3$ in Proposition~\ref{prop:MM+LC}, and
$S^*_2$ contains the following network: 
% We point out that the set $S^*_n$ is non-empty for all $n$. For $n\geq 3$, we provide a subset of $S_n^*$ in . Bimolecular networks in one species do not have nondegenerate multistationarity {\color{orange} CITE?? THE OTHER PAPER??}. The following network belongs to $S^*_2$:
 \begin{align*}
      &
      \left\{
      A+B \stackrel[1/4]{1/32}{\leftrightarrows} 2A, \quad 2B \stackrel[1/4]{1}{\leftrightarrows} A, \quad \emptyset \stackrel[1]{1}{\leftrightarrows} B \right\} ~.
 \end{align*}
Indeed, the indicated rate constants generate a system with $3$ nondegenerate positive steady states  -- with approximate values 
$(0.419694, 1.11107)$, 
$(2.65005, 2.3128)$, and
$(216.681, 27.5757)$  
 -- 
and so this system is nondegenerately multistationarity and also does not have ACR. 
\end{remark}

\begin{theorem}[Thresholds for full-dimensionality, multistationarity, and non-ACR] \label{theorem:threshold}
Consider the setup for generating random reaction networks $G_n$, described in Section~\ref{sec:random-setup}, for some choice of edge probabilities.  Then 
there exist threshold functions (``thresholds'', for short) $0<r_0(n)\leq r_1(n)\leq r_2(n)$, such that for any $\{ p_n \}_{n \geq 1}$:
\begin{itemize}
\item[(0)] If $p_n\gg r_0(n)$, then 
	$G_n$ is full-dimensional w.h.p. %when $p_n\gg r_0(n)$. 
\item[(1)] If $p_n\gg r_1(n)$, then 
	$G_n$ is full-dimensional and nondegenerately multistationary w.h.p. 
\item[(2)] If  $p_n\gg r_2(n)$, then 
	$G_n$ is full-dimensional, is nondegenerately multistationary, and does {\bf not} have unconditional ACR (in any species) w.h.p.
\end{itemize}
\end{theorem}

\begin{proof}
Being a full-dimensional network is a monotone increasing property with respect to adding reactions (with no new species). This fact, combined with a well-known result from 
 the theory of threshold functions~\cite{threshold}, proves part~(0).
 
From Remark~\ref{rem:monotone}, 
the property (for full-dimensional networks) of being nondegenerately multistationary is monotonically increasing with respect to adding reactions (with no new species). 
Exploiting again the theory of threshold functions~\cite{threshold}, we obtain part~(1).
 
Finally, Lemma~\ref{lem:monotone}(ii) and the theory of threshold functions together imply that there exists a threshold function $r_2(n)$  for $G_n$ to contain a subnetwork $H \in S^*_n$.  This implies part~(2).
%let $r_2(n)$ be the threshold for $G_n$ to contain a subnetwork in $S^*_n$, then from  
% we obtain the proof of part~(2).
 %the remaining two parts.
\end{proof}

%{\color{violet} Above proof reads like having ACR is also monotonically increasing wrt more edges. But that's false, right? It took me a while to realise we mean both together. Rephrase, maybe?}
%
%{\color{blue} Added a word.}

%\badalnote{either rephrase (2) or say that the threshold for full-dim and no ACR is the same as the threshold for full dim, nondeg mult, and no ACR.}

Theorem \ref{theorem:threshold} 
implies that when a random network is sufficiently dense, it is multistationary w.h.p. (after a threshold $r_1(n)$) and also lacks unconditional ACR (after a threshold $r_2(n)$).  However, computing these thresholds is generally difficult, 
because it is challenging to determine whether a large reaction network is multistationary and whether it precludes ACR.  In fact, while there are sufficient conditions for ACR, such as the Shinar-Feinberg criterion~\cite{ACR}, 
there are no easy-to-check necessary conditions for ACR (for general networks)~\cite[Section~2]{MST}.

Nevertheless, there is a fruitful strategy for establishing upper bounds on the thresholds $r_1(n)$ and $r_2(n)$, which we describe in detail in the remainder of this subsection.  
%{\color{red} The idea of the next sentence already appeared in Prop 2.7, so rephrase?} 
The underlying idea comes 
from the fact (stated earlier in Lemma~\ref{lem:monotone}) 
%Recall from PROP 2.7 (CITE)
that multistationarity can sometimes be \textit{lifted} from a small subnetwork to the whole network. 
Therefore, in lieu of determining when multistationarity of the entire network emerges (as edge probabilities increase), we instead investigate when a small multistationary subnetwork  emerges. 
The choice of edge probabilities dictates which such subnetworks emerge first.
%{\color{brown} \sout{
%It depends on the choice of edge probabilities which of the %small multistationary 
%subnetworks appears first.}} %\sout{\red{way we generate random reaction networks (i.e. the choice of edge probabilities). }}
For the  edge probabilities we consider in the next section, we focus on a particular multistationary subnetwork, as follows. % defined next.

\begin{definition}[Sets $S_{M,n}$ of multistationary motifs] \label{M1}
For $n \in \mathbb{Z}_{>0}$, let $S_{M,n}$ denote the set of all networks of the following form:
\begin{align} \label{eq:multi-motif}
    &
    \{ X_i \leftrightarrows X_j + X_k, \quad 0 \leftrightarrows X_i, \quad 0 \leftrightarrows X_j, \quad X_k \leftrightarrows 2X_k  \} ~,
\end{align}
where $i,j,k$ are distinct indices with $1 \leq i,j,k \leq n$.  
%Consider the network $M$ with 3 species $X_i,X_j,X_k$
Each network~\eqref{eq:multi-motif} is a {\em multistationary motif}.
\end{definition}

\begin{remark} \label{rem:motif-has-mss}
Recall from Example~\ref{eg:multimotifrates} and Example ~\ref{eg:multimotifrates} (continued) that each multistationary motif~\eqref{eq:multi-motif} is 
full-dimensional (3-dimensional) and
nondegenerately multistationary, and does not have unconditional ACR (in any species).  
%{\color{red} Careful: the ``no ACR'' statement is in the continuation of that example.  Adjust?}
\end{remark}
%We denote by~$S_{M, n}$ the set of all networks that are equivalent to the above network, up to species labels \textcolor{blue}{among $\{X_1,\dots,X_n\}$} and we refer to elements of $S_{M,n}$ as a {\em multistationary motif}. 

%Recall that each multistationary motif~\eqref{eq:multi-motif} is multistationary and lacks ACR. 
Our next aim is to show (in Proposition~\ref{prop:MM+LC} below) how to join a multistationary motif to another network (which we call a ``lifting component'') so that the resulting network again is 
full-dimensional and
nondegenerately multistationary, and lacks unconditional ACR.
%multistationary and lacks ACR, but also is full-dimensional.
% is full-dimensional, has nondegenerate multistationarity, and still does not exhibit ACR. 
%Denote the threshold of such a lifting component by $p_2'$, then $\max\{p_1',p_2'\}$ gives the desired upper bound. Hence, for $p\gg \max\{p_1',p_2'\}$, a random network admits nondegenerate multistationarity and has no ACR w.h.p.

%\textcolor{blue}{While part $(1)$ of Theorem \ref{theorem:threshold} deals with the emergence of multistationarity, part $(2)$  concerns the coexistence of multistationarity and ACR. Specifically, part $(2)$} implies that when a random network is sufficiently dense, multistationarity and ACR are mutually exclusive w.h.p. While there are sufficient conditions like Shinar-Feinberg criterion (\cite{ACR}) that ensures ACR in a network, conditions which completely rule out ACR for large networks are not known. Hence, finding the threshold $r_2(n)$ is also a challenging task. \textcolor{blue}{We will discuss an approach that can produce an upper bound for $r_2(n)$}.  

\begin{definition}[Sets $S_{L,k}$ of  lifting components]
For $k \in \mathbb{Z}_{>0}$, let $S_{L,k}$ be the set of all reversible reaction networks  for which the associated graph is a tree on $k$ vertices (that is, complexes), where each of the vertices has the form $X_i$ (for $i \in \mathbb{Z}_{>0}$). 
%that have exactly $k$ distinct complexes of the form $X_i$, and whose associated graph is a spanning tree. 
Every network in $S_{L,k}$ is a {\em lifting component}.
\end{definition}

\begin{example} An example of a network (lifting component) in 
$S_{L,3}$ is $
\{ X_2\leftrightarrows X_3 \leftrightarrows X_4 \}
$.
%
%{\color{violet} \sout{
%		$S_{L,k}$ is $
%		\{ X_1\leftrightarrows X_2 \leftrightarrows \dots \leftrightarrows X_k \}
%		$.}}
\end{example}

%\sout{
%It is straightforward to check that each lifting component in $S_{L,k}$ has dimension $k-1$.  On the other hand, each multistationary motif~\eqref{eq:multi-motif} has dimension $3$.  We can therefore build a full-dimensional ($n$-dimensional) network by joining, in a certain way, a multistationary motif to a lifting component with $n-2$ species (that is, in $S_{L,n-2}$).  We describe a set of such networks next.}
Next, we describe a set of networks obtained by joining a multistationary motif~\eqref{eq:multi-motif}, which has dimension $3$, to a lifting component of dimension $n-3$, so that the result is full-dimensional.

\begin{definition}[Sets $S_{J,n}$ of joined multistationary motifs and lifting components] \label{def:joined-sets}
For $n \geq 3$, let $S_{J,n}$ be the set of all reaction networks whose reactions can be written as the union of a network $G_M \in S_{M,n}$ and a network $G_L\in S_{L,n-2}$, such that $G_M$ and $G_L$ have exactly one 
species in common. 
\end{definition}

An example of such a joined network is shown later (see Figure \ref{fig3} and Example \ref{example:dense}).

%\textcolor{blue}{The next theorem specifies how we add a lifting component to a multistationary motif to obtain a full-dimensional network with multistationarity and no ACR species. Since a multistationary motif in $S_M$ has dimension 3, to get full-dimensionality, the smallest lifting component we can add must have $n-2$ species, one of which is a species in the multistationary motif.}

\begin{proposition}[Properties of  $S_{J,n}$] \label{prop:MM+LC}
 For $ n \geq 3$, the following hold: 
 \begin{enumerate}
 	\item $S_{J,n} \subseteq S^*_n$. Consequently, every network $H \in S_{J,n}$ is full-dimensional and nondegenerately multistationary, and does not have unconditional ACR (in any species). 
 	\item If some $H \in S_{J,n}$ is a subnetwork of a network $G$ with $n$ species, then $G$ is full-dimensional and nondegenerately multistationary, and does not have unconditional ACR (in any species).
\end{enumerate}
%Let $G_M$ and $G_L$ be reaction networks in $S_{M,n}$ and $S_{L,n-2}$, respectively, such that $G_M$ and $G_L$ have exactly one 
%species in common. 
%Let $G$ be the reaction network whose set of reactions is the union of the reactions of $G_M$ and $G_L$. 
%Then $G$ 
% Every network $G \in S_{J,n}$ 
% is full-dimensional and nondegenerately multistationary, and has no unconditional ACR (in any species).
\end{proposition}

\begin{proof} 
We first prove part (1). Let $H \in S_{J,n}$.  By definition, there exist $G_M \in S_{M,n}$ and  $G_L \in S_{L,n-2}$, with exactly one 
species in common, such that the reactions of $H$ are a union of those in $G_M$ and $G_L$.  
Relabel the species of $H$, if needed, so that $G_M$ is the following network:
%Without loss of generality, suppose that $G_M$ contains the species $X_1,X_2,X_3$
\[
\{ X_1 \leftrightarrows X_2 + X_3, \quad 0 \leftrightarrows X_1, \quad 0 \leftrightarrows X_2, \quad X_3 \leftrightarrows 2X_3 \}~,
\]
and also that the species of $G_L$ are 
$X_{\ell}, X_4, X_5, \dots,X_n$, for some $\ell \in\{1,2,3\}$.  

It is straightforward to check that $H$ is full-dimensional (that is, has dimension $n$). 
Thus, to show that $H\in S_n^*$, it suffices to show that there exists a vector of rate constants $\kappa$ such that $(H,\kappa)$ is nondegenerately multistationary and does not have ACR in any species.  Accordingly, we define $\kappa$ as follows.
First, we choose the rate constants for reactions in $G_M$ 
as in~\eqref{eq:motif-with-rate-constants} in Example~\ref{eg:multimotifrates}, so that $G_M$ has three nondegenerate, positive steady states: 
	$(13,20,1),(18,15,2),(21,12,3)$.
%such that $G_M$ has 3 positive steady states $(x_{1j},x_{2j},x_{3j})$ for $j=1,2,3$ while having no ACR in $X_1,X_2,X_3$.  
Next, fix all rate constants for reactions in $G_L$ to be~$1$. 
Using the fact that $G_L$ is a spanning tree,
a simple computation shows 
%it is straightforward to check 
that the positive steady states are 
{
$(x_\ell, x_4,  x_5, \dots, x_n) = (c,c,c,\dots, c)$,} 
where $c$ is any positive real number, and these steady states are all nondegenerate.

{
We consider first the case when the common species is $X_{\ell}= X_3$.
}
We claim that, {in this case,} the following are nondegenerate steady states of $(H,\kappa)$:
	\begin{align} \label{eq:steady-states-joined-network}
	(13,20,1, 1, \dots, 1),\quad (18,15,2, 2, \dots, 2), \quad (21,12,3, 3, \dots, 3)~.
 	\end{align}
To see this, 
let $\dot{x_i}=f_i$ for $i=1,2,3$, and 
$\dot{x_i}=g_i$ for $i=3,4,\dots,n$, denote the ODEs for $G_M$ and $G_L$, respectively, with rate constants as defined above.  Hence, 
	$(13,20,1),(18,15,2),(21,12,3)$ satisfy $f_1=f_2=f_3=0$, and $(1,1 \dots, 1),(2, 2, \dots, 2),  ( 3, 3, \dots, 3)$ satisfy $g_3=g_4=\dots = g_n=0$.  
	Next, the ODEs of $(H, \kappa)$ are as follows:
	\begin{align*}
	 \dot{x_i} ~&=~ f_i \quad \quad  \quad {\rm for }~  i=1,2 \\
	  \dot{x_3}~&=~ f_3+g_3 \\
	  \dot{x_j}~&=~ g_j \quad \quad \quad {\rm  for }~ j=4,\dots,n.
	  \end{align*} 
	% $\dot{x_i}=f_i$ for $i=1,2$, $\dot{x_3}=f_3+g_3$ and $\dot{x_i}=g_i$ for $i=4,\dots,n$, 
	 %and so
	 Hence, the vectors in~\eqref{eq:steady-states-joined-network} indeed are steady states of $(H, \kappa)$.

To show that the steady states~\eqref{eq:steady-states-joined-network} are nondegenerate, 
we must show that the Jacobian matrix of $(H,\kappa)$, when evaluated at each of these steady states, is nonsingular (recall that $H$ is full-dimensional). 
%Since $G$ is full-dimensional, to show that the steady states of $(G,\kappa)$ are nondegenerate, we must show the Jacobian matrix evaluated at these steady states are non-singular. 
Since $G_L$ has mass conservation among the species $X_3,X_4,\dots,X_n$, we have $g_3+g_4+\dots+g_n=0$. Adding rows $4, 5, \dots, n$ to row $3$ of the Jacobian matrix yields a triangular block matrix $\begin{bmatrix} A & 0\\ B&C \end{bmatrix}$, where $A$ is the Jacobian matrix of $G_M$
and $C$ is obtained from the Jacobian matrix of $G_L$ 
by setting $x_3$ to 0. 
%where $X_3$ is replaced by $0$. 
As both $A$ and $C$ are nonsingular, 
when evaluated at any of the positive steady states of its corresponding system, 
we conclude that the Jacobian matrix of $(H,\kappa)$, 
when evaluated at any of the  
steady states~\eqref{eq:steady-states-joined-network},
 is nonsingular, as desired.  
 %{\color{blue} Not sure when $n=3$ case needs to be handled separately.}
%Next, we look at the steady states for $G=G_M\cup G_L$. We claim that $G$ also has 3 positive steady states: $\mathbf{x_j}=(x_{1j},x_{2j},x_{3j},x_{3j},\dots,x_{3j})$ for $j=1,2,3$. Suppose $G_M$ has ODEs $\dot{x_i}=f_i$ for $i=1,2,3$ and $G_L$ has ODEs $\dot{x_i}=g_i$ for $i=3,4,\dots,n$. Then the ODEs for $G$ are $\dot{x_i}=f_i$ for $i=1,2$, $\dot{x_3}=f_3+g_3$ and $\dot{x_i}=g_i$ for $i=4,\dots,n$. Since $\mathbf{x_j}$ satisfy $f_i=0$ for $i=1,2,3$ and $g_i=0$ for $i=3,\dots,n$, they are positive steady states of $G$. This observation, together with the fact that $G_M$ has no ACR in any species, implies that $G$ also has no ACR in any species.

By inspection of the steady states~\eqref{eq:steady-states-joined-network}, we see that there is no ACR in any species. As for the remaining cases, when $X_\ell$ (the common species of $G_M$ and $G_L$) is $X_1$ or $X_2$, the argument is very similar to what is shown above and so the result holds in those cases, too.

Finally, part (2) follows directly from part (1) and Lemma \ref{lem:monotone}.
%Note that while we chose $X_3$ as the common species of $G_M$ and $G_L$, the above arguments and hence, the result also hold for the choice of $X_1$ and $X_2$.}
\end{proof}

% COROLLARY
\begin{corollary}[Bound on threshold $r_2$] \label{cor:thresholds-equal}
Consider the setup for generating random reaction networks $G_n$, described in Section~\ref{sec:random-setup}, for some choice of edge probabilities.
Let $r_2(n)$ be a
%{\color{green} the} 
threshold as defined in Theorem~\ref{theorem:threshold}, and let 
$r_2'(n)$ be the threshold for $G_n$ to contain, as a subnetwork, some $H \in S_{J,n}$. If $p_n \gg r_2'(n)$, then $G_n$ is full-dimensional, is nondegenerately multistationary, and does not have unconditional ACR (in any species) w.h.p.
Consequently, $\limsup \frac{r_2(n)}{r_2'(n)} < \infty$.
% $r_1'(n)$ and $r_2'(n)$ be
% %{\color{green} the}  
% thresholds for $G_n$ to contain as a subnetwork a multistationary motif from $S_{M, n}$ and a lifting component in $S_{L,n-2}$, respectively. 
% Assume that $\{p_n\}_{n \geq 3}$ satisfies the following:
% 	\begin{enumerate}
% 	\item $p_n \gg \max\{r_1'(n),r_2'(n)\}$, and
% 	\item $G_n$ contains some $G \in S_{J,n}$  
% 	as a subnetwork, w.h.p.
% 	%	a multistationary motif from $S_{M, n}$ and a lifting component in $S_{L,n-2}$ that share exactly one species, w.h.p.
% 	%	a multistationary motif and a lifting component sharing exactly one species, w.h.p.
% 	\end{enumerate}

%{\color{violet} [Anne to Tung: I like the restatement!  I'm puzzled about the inequality in teal.  If the two thresholds are the same, then it could be that the inequality is violated, right?  It could be, e.g. $r_2(n) = r_2'(n) +1$.]}
%{\color{violet} [Anne: I like the update.  Is there still some subtlety, if e.g., $r_2(n)=n$ and $r_2'(n)$ is $n$ or $n^2$ based on whether $n$ is even or odd?  (Of course, these are not really thresholds, but I think my idea is conveyed.)] }
\end{corollary}
%In such a case, $\max\{r_1'(n),r_2'(n)\}$ is an upper bound for the threshold $r_2(n)$. 

\begin{proof}
This result follows directly from Proposition~\ref{prop:MM+LC} and Theorem~\ref{theorem:threshold}.  
%{\color{blue} Please double check this proof and then remove this comment!}
% {\color{orange} [May need to update, because proposition on monotonicity was edited.]}{\color{blue}[Should be fine now with the stuff that I added: Def 3.6, proof of 3.7 and prop 3.13.]}
%This result follows directly from Propositions~\ref{lem:monotone}
%and~\ref{prop:MM+LC}, the fact that full-dimensionality is a monotone property with respect to adding reactions (without new species), and 
%the definition of $r_2(n)$ (in Theorem~\ref{theorem:threshold}). 
\end{proof}

%{\color{orange} Are these last two results used in the next section?} {\color{blue} Yes, it is used in Proposition 4.15.}

%{\color{brown} \sout{
%In the next section, we explicitly compute the %threshold upper-bound 
%upper bound on $r_2(n)$ 
% in 
%Corollary~\ref{cor:thresholds-equal}
%for a specific choice of edge probabilities, and show that, for that choice, the bound is tight.%actually {\color{blue} (asymptotically)} equals the threshold.
%}}

\begin{remark} \label{rem:generic-motif} 
In the next section, we show that for a certain choice of 
edge probabilities, the thresholds for full-dimensionality, multistationarity, and containing (as a subnetwork) a multistationary motif in $S_{M,n}$ coincide. 
That is, in this scenario, $r_0(n) = r_1 (n) = r_1' (n)$, where $r_1'$ denotes the threshold for $G_n$ to contain a multistationary motif. 
% to appear as a subnetwork of $G_n$.  
(In fact, in this setting, once we pass the threshold for full-dimensionality, a multistationary motif emerges and furthermore its multistationarity can be lifted; see Theorem~\ref{thm:maintheorem}.)
Intuitively, the reason a motif in $S_{M, n}$ is among the first small multistationary subnetworks to emerge is that it is fairly ``generic": it contains only a pair of reversible reactions and some flow reactions involving those species. Meanwhile, other small multistationary networks, such as the following (from~\cite{atoms_multistationarity}):
\begin{align*} \{
A \leftrightarrows A+B \leftrightarrows 2A, \quad\quad 0 \leftrightarrows A, 
\quad\quad  0\leftrightarrows B \}~,
\end{align*} 
may contain specific pathways such as $A \leftrightarrows A+B \leftrightarrows 2A$, where a species (here, $A$) must appear in all three complexes, and therefore such networks
are expected to emerge at higher thresholds.
\end{remark}

%While there are many small networks which could serve as multistationary motifs, there is a reason to consider the motif in Definition~\ref{M1}. Firstly, it is optimal to find a multistationary motif whose threshold for appearing in the network is same as the threshold for multistationarity. Later we will prove that this is indeed the case (i.e. $p_1'(n)=p_1(n)=p_0(n)$) for some special choice of edge probabilities  (cf. Corollary~\ref{cor:mss}). In a sense, this motif is ``generic"  since it contains a single pair of reversible reactions, and the inflows and outflows of the species. On the other hand, many other small multistationary networks, such as 
%\begin{align}\label{eqn:motif2}
%A \leftrightarrows A+B \leftrightarrows 2A \quad\quad 0 \leftrightarrows A 
%\quad\quad  0\leftrightarrows B 
%\end{align}

%Second, biologically speaking, $M$ is a fairly generic motif. A pair of reversible reactions and the inflows and outflows of the 3 species may easily happen by chance. On the other hand, the motif~(\ref{eqn:motif2}) requires a more specific pathway: $A \leftrightarrows A+B \leftrightarrows 2A$, which may have special purpose or design and is less likely to appear by chance. 

%We compute this upper  bound, $\max\{r_1'(n),r_2'(n)\}$, explicitly for a specific choice of edge probabilities in Section~\ref{sec:stochasticmodel}. \textcolor{blue}{Additionally, we show with that choice of edge probabilities, the upper bound in fact coincides with the actual threshold $r_2(n)$.} 

% -----------------------------
% Stochastic block model
% -----------------------------

\section{Multistationarity and ACR in type-homogeneous stochastic block model}\label{sec:stochastic_block}
The prior section considered general random graph models without specifying the edge probabilities. 
In this subsection,
we introduce a specific choice of edge probabilities (Section~\ref{sec:stochasticmodel}) and then compute the resulting thresholds 
%for full-dimensionality, multistationarity, and lack of ACR, as in 
from Theorem~\ref{theorem:threshold} (see Theorem~\ref{thm:maintheorem} in Section~\ref{sec:threshold-block}).

\subsection{A stochastic block model}\label{sec:stochasticmodel} %for bimolecular reaction networks

One possible choice of edge probabilities comes from the Erd\H os-R\'enyi random graph model; here, the edge probabilities are uniform 
(that is, every edge is equally likely).   
In this framework, reactions of the form $X_i+X_j\leftrightarrows X_h+X_k$ are overwhelmingly the most prominent~\cite{prevalence}. 
However, this situation is unlikely to occur in biochemistry. 
Indeed, in applications, one expects to see various types of complexes and reactions, such as inflow and outflow reactions $0 \leftrightarrows X_i$ or association and disassociation reactions $X_i+X_j\leftrightarrows X_k$. 
%\badalnote{is it possible to replace all ``leftrightarrows'' with ``rightleftarrows'' which is how people write on board? Of course, rate constant labels will have to be switched, which may be too much to do.}

Therefore, we instead consider a model 
 in which {\em reaction types} 
 %various types of  reactions (edges) 
 are equally represented. 
 %\badalnote{suggestion: replace ``various types of  reactions (edges)'' with ``{\em reaction types}''}
This model is a specific case of the {\em stochastic block models}~\cite{holland1983stochastic} 
introduced in \cite{prevalence_block}. 
% {\color{blue}(see \cite{holland1983stochastic} for detailed exposition)}  
To define this model, we need the following partitions of sets of vertices and edges: % (of the random graphs considered in the prior subsection).

%We first define a partition on the set of all possible vertices and edges as follows. 

\begin{definition}[$C_i$ and $E_{i,j}$] \label{def:partitionvertexedge}
	Let $n \geq 1$.  Consider the following partition of the set of vertices $V_n$, as in~\eqref{eqn:vertices}, into 3 subsets: 
%	The set of vertices in $V_n$ is partitioned into the following 3 sets:
	%according to the number of vertices in each type:
	\begin{enumerate}
		\item $C_0= \{ 0 \}$, %; hence \quad $|C_0|=1$
		\item $C_1= \{aX_i \mid 1\leq i\leq n\ \text{ and } a\in\{1,2\}\}$, %; hence \quad $|C_1|\sim n$
		\item $C_2= \{X_i+X_j \mid 1 \le i, j \le n {\rm ~and~} i\neq j\}$. %; hence \quad $|C_2|\sim n^2$
	\end{enumerate}
	%For $0\leq i \leq j \leq 2$, let 
	Let $E_{i,j}$ denote the set of (undirected) edges $(u,v)$ with $u \in C_i$ and $v\in C_j$; in particular: 
\begin{enumerate}
\item $E_{0,1} = \{0\leftrightarrows aX_i \mid 1\leq i\leq n \text{ and } a \in \{1,2\} \}$,
\item $E_{0,2}=\{0\leftrightarrows X_i+X_j \mid 1\leq i,j \leq n \text{ and } i\neq j\}$,
\item $E_{1,1}=\{aX_i \leftrightarrows bX_j \mid 1\leq i,j \leq n,~ a,b \in \{1,2\}, \text{ and } (a,i)\neq(b,j)\}$,
\item $E_{1,2}=\{aX_i\leftrightarrows X_j+X_k \mid 1\leq i,j,k\leq n,~ a \in \{1,2\}, \text{ and } j\neq k\}$,
\item $E_{2,2}=\{X_i+X_j\leftrightarrows X_k+X_h \mid 1\leq i,j,k,h\leq n \text{ and } i\neq j, k\neq h, \text{ and } 
		(i,j)\neq (k,h) \neq (j,i) \}$. %, (i,j)\neq (h,k)\}$,
\end{enumerate}
%where the coefficients $a,b \in \{1,2\}$. 
%We also refer to the sets $E_{i,j}$ as the ``type" of reactions.
Two reactions in the same set $E_{i,j}$ have the same {\em type.}
\end{definition}

\begin{remark} \label{rem:size-of-parts}
	The sets $E_{0,1}, E_{0,2}, E_{1,1}, E_{1,2}, E_{2,2}$ partition the set of all possible edges of a graph with vertex set $V_n$.
	Also, $|C_0|=1$, $|C_1|\sim n$, and $|C_2|\sim n^2$. So, $|E_{i,j}| \sim n^{i+j}$ for $0 \leq i \leq j \leq 2$.  %$0\leq i\leq 2, \text{ and } 1\leq j\leq 2$.
\end{remark}

In what follows, we denote the minimum of two numbers $a$ and $b$ as follows:
\[
	a \wedge b ~:= \mathrm{min}(a,b)~.
\]

\begin{definition} \label{def:stochastic-block-model}
The {\em type-homogeneous stochastic block model} 
generates random graphs $G(V_n,p_n)$ with vertex set $V_n$ (as described in 
Definition~\ref{def:rdm-graph-general}) 
%Section~\ref{sec:random-setup}, 
via
the following choice of edge probabilities:
%and the probability that there is an edge from $E_{i,j}$ is given by 
\begin{equation}\label{edge_probability}
\textrm{if } e \in E_{i,j}, \textrm{ then }
\P(e \text{ is an edge of } G(V_n,p_n))
		~=~	n^{4-i-j}p_n\wedge 1~,
\end{equation}
for all  $0 \leq i\leq j \leq 2$.
\end{definition}

\begin{remark}\label{rem:edgefunc}
	In Definition~\ref{def:stochastic-block-model}, 
	%for type-homogeneous stochastic block model 
	for vertices $u\in C_i $ and $ v\in C_j$, 
	the edge probability function (as in Definition~\ref{def:edgeprob}) for the edge $e=(u,v)$ is given by
%	the edge probabilities are monotonically increasing and given by: 
%	\[\phi_{\{u,v\}}(p_n):=	n^{4-i-j}p_n\wedge 1 \text{ when }u\in C_i \text{ and } v\in C_j.\]
$	\phi_{ e }(p_n) =	n^{4-i-j} p_n\wedge 1 $.
This edge probability function is readily seen to be non-decreasing in $p_n$.
\end{remark}

Recall from Definition~\ref{def:rdm-network-general} 
%Section \ref{sec:random-setup} 
that each random graph $G(V_n, p_n)$ generates a random reaction network~$G_n$. % as follows: 
%{\color{blue} [These 3 bullet points are repeated from earlier.  Why?]}
%\begin{itemize}
%\item The set of species of $G_n$ is $\{X_1,X_2,\dots ,X_n\}$.
%\item The complexes of $G_n$ are the non-isolated vertices of $G(V_n, p_n)$.
%%Each vertex with positive degree in $G(V_n,p_n)$ is a complex in $G_n$. 
%\item The (reversible) reactions of $G_n$ correspond to the edges of $G(V_n, p_n)$.
%%Each edge in $G(V_n,p_n)$ \textcolor{blue}{generates a corresponding reversible pair} of reaction in $G_n$.
%\end{itemize}
%
%{\color{violet} We didn't define degree for graphs anywhere.
%	
%[optional remark]	It is a matter of taste but I would have written it the other way around since we are talking about generating Gn from G(V,pn): something like edges of G(V,pn) gives reversible reactions of Gn. [NK]}
The edge probabilities~\eqref{edge_probability} ensure that, in $G_n$, 
%all types of reactions  appear in $G_n$ in similar abundances (by Remark~\ref{rem:size-of-parts}). In particular, 
the expected numbers of edges of each type are of the same magnitude (namely, $\sim n^4p_n$), whenever possible.  

\begin{example} \label{ex:no-mss}
When $p_n=\frac{1}{n^{3.5}}$, the expected number of reactions of each type in $G_n$ is $\sim\sqrt{n}$. 
%With this choice of $p_n$, $G_n$ not multistationary w.h.p. based on Theorem \ref{thm:maintheorem} in the next subsection.
\end{example}

\begin{example} \label{ex:2.9}
When $p_n=\frac{1}{n^{2.9}}$, $G_n$ contains all reactions in $E_{0,1}$ (since $n^{4-0-1} p_n \wedge 1 =1$), 
and the expected number of reactions for each of the remaining types is $\sim n^{1.1}$.  
%the expected number of reactions of each type except for $E_{0,1}$ is $\sim n^{1.1}$ while all reactions in $E_{0,1}$ are present. With this choice of $p_n$, $G_n$ is multistationary and does not have unconditional ACR w.h.p based on Theorem \ref{thm:maintheorem} in the next subsection.
\end{example}

\begin{remark} \label{rem:all-flows}
When 
$p_n \gg \frac{1}{n^3}$ (for instance, $p_n=\frac{1}{n^{2.9}}$ in Example~\ref{ex:2.9}), 
$G_n$ contains all reactions in $E_{0,1}$, including all inflows/outflows $0\leftrightarrows X_i$, and so $G_n$ is full-dimensional. 
\end{remark}

%then the edge probability in $E_{0,1}$ is $n^3p_n\wedge 1 = 1$, so all inflows and outflows of the form $0\leftrightarrows X_i$ are present. 
%Thus, a random network in this regime is full-dimensional. This fact will be used for several lifting arguments in subsequent propositions and lemmas. 

%\begin{example}
%When $p_n\gg\frac{1}{n^3}$, the expected number of edges of $G_n$ from each $E_{i,j}$ -- except $E_{0,1}$ -- is $n^4p_n$, which is greater than the total number of edges in $E_{0,1}$. In this case, all edges in $E_{0,1}$ are present in $G_n$.
%\end{example}

%However, this is sometimes not possible.  For instance, when $p_n\gg\frac{1}{n^3}$, the expected number of edges of $G_n$ from each $E_{i,j}$ -- except $E_{0,1}$ -- is $n^4p_n$, which is greater than the total number of edges in $E_{0,1}$. In this case, all edges in $E_{0,1}$ are present in $G_n$.

%\begin{example}
%When $p_n=\frac{1}{n^{2.9}}$, the expected number of reactions of each type except for $E_{0,1}$ is $\sim n^{1.1}$ while all reactions in $E_{0,1}$ are present. With this choice of $p_n$, $G_n$ is multistationary and does not have unconditional ACR w.h.p based on Theorem \ref{thm:maintheorem} in the next subsection.
%
%\end{example}

In the next subsection, we see that the 
choice of $p_n$ in Example~\ref{ex:no-mss} generates networks $G_n$ that are  
not multistationary w.h.p.,
while the choice of $p_n$ in Example~\ref{ex:2.9} generates networks that are multistationary and lack ACR w.h.p.\ (see Theorem \ref{thm:maintheorem}).

%-----------------------
% SUBSECTION - THRESHOLDS for stochastic block
%-----------------------
\subsection{Thresholds for multistationarity and ACR} \label{sec:threshold-block}
For the type-homogeneous stochastic block model, 
the thresholds for nondegenerate multistationarity and (no) ACR are stated in the following theorem (which is proven later in Section~\ref{sec:proof}). %later in this section).  
\begin{theorem}[Type-homogeneous stochastic block model] \label{thm:maintheorem} 
Consider the setup for generating random reaction networks $G_n$, described in Section~\ref{sec:random-setup}, for the edge probabilities given by \eqref{edge_probability}.  Then,  for any $\{ p_n \}_{n \geq 1}$:
%Consider a random reaction network $G_n$, with edge probabilities given by \eqref{edge_probability}.
	\begin{enumerate}[(i)]
		\item (Sparse regime) If $\ds \frac{1}{n^4}\ll  p_n\ll \frac{1}{n^{10/3}}$, then w.h.p. 
			$G_n$ has deficiency zero, 
			is not multistationary, 
			and has unconditional ACR (in some species).
		\item (Dense regime, window of co-existence) If $\ds \frac{1}{n^3}\ll p_n \leq \frac{\frac{2}{17}\log(n)-c(n)}{n^3}$ for some $c(n)\to\infty$, then w.h.p $G_n$ is nondegenerately multistationary and has unconditional ACR in some species.
		\item (Dense regime) If $\ds p_n \geq \frac{\log(n-2)+c(n)}{n^2(n-2)}$ for some $c(n)\to\infty$, then w.h.p. $G_n$ is nondegenerately multistationary and does {\bf not} have  unconditional ACR (in any species).
	\end{enumerate}
\end{theorem}

\begin{remark} \label{rem:no_reactions}
When $p_n \ll \frac{1}{n^4}$, the expected number of reactions in $G_n$ is $\ll 1$. 
Accordingly, we do not consider this interval in Theorem \ref{thm:maintheorem}.
%Thus, this interval is not included in the sparse regime in Theorem \ref{thm:maintheorem} to make sure that a random network in the sparse regime is not empty (i.e. has at least one reaction) w.h.p.
\end{remark}

\begin{remark}
Theorem~\ref{thm:maintheorem} immediately yields the following thresholds $r_0(n),r_1(n),r_2(n)$ (as defined in Theorem~\ref{theorem:threshold}) for networks generated by the type-homogeneous stochastic block model:
	\[  r_0(n)~=~ r_1(n)~=~ \frac{1}{n^3} \quad {\rm and} \quad r_2(n)~=~\frac{\log(n)}{n^3}~.\]
%In particular, $r_0(n)=r_1(n)=\frac{1}{n^3}$, and $r_2(n)=\frac{\log(n)}{n^3}$.
\end{remark}

%{\color{violet} Is the above statement the correct implication? I think we obtain (compute) the thresholds to give the theorem. I don't think we obtain the thresholds from the theorem.[NK]}
%
%{\color{blue} The theorem contains stronger results than thresholds (thresholds only require $\gg$ and $\ll$), so the thresholds should be an implication of the theorem. E.g. when $p_n \gg \frac{\log(n)}{n^3}$ clearly it satisfies case 3 and thus we have a threshold for non-ACR.}

\begin{remark}[Thresholds via number of reactions] \label{remark:thresholds}
Theorem \ref{thm:maintheorem} can be rephrased in terms of the expected number of edges (reactions) instead of edge-probability thresholds, as follows. 
For random reaction networks $G_n$ (with $n$ species) generated by the type-homogeneous stochastic block model, 
the following are implied directly by Theorem \ref{thm:maintheorem}:
\begin{enumerate}[(i)]
    \item If the expected number of reactions of each type is $\gg 1$ and $\ll n^{2/3}$, then $G_n$ has deficiency zero and thus is not multistationary w.h.p.
    \item If the expected number of reactions of each type is $\gg n$ but less than $\frac{2}{17}n(\log(n)-c(n))$ for some $c(n)\to\infty$, 
	then  $G_n$
    is nondegenerately multistationariy and has unconditional ACR in some species w.h.p.
    \item  If the expected number of reactions of each type is greater than $n(\log(n)+c(n))$  for some $c(n)\to\infty$, then 
 $G_n$ 
	 is nondegenerately multistationary and does {\bf not} have unconditional ACR (in any species) w.h.p.
\end{enumerate} 
\end{remark}

% --------------------------
% TWO EXAMPLES
% --------------------------

\begin{figure}[htb]
\includegraphics[scale=0.8]{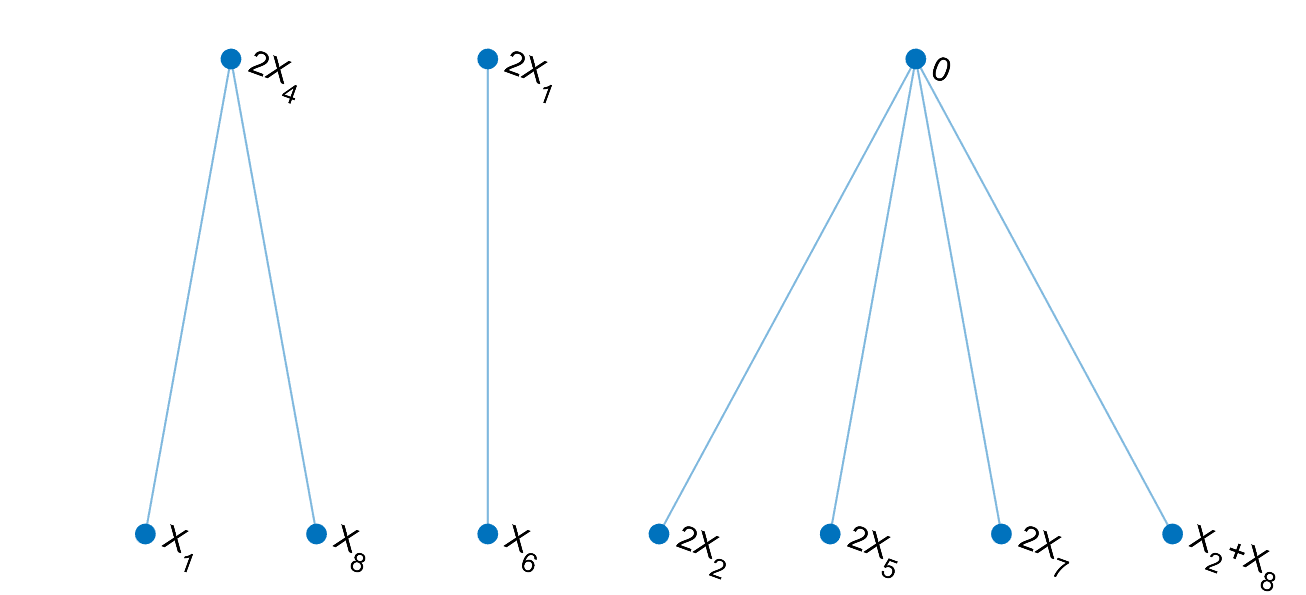}
\caption{ 
A realization of a random reaction network generated by the type-homogeneous stochastic block model 
in the sparse regime ($n=8$ species and $ p_n=\frac{0.5}{n^{3.5}}$).  Edges represent reversible reactions. } 
\label{fig1}
\end{figure}

% EXAMPLE STARTS HERE
\begin{example}[Sparse regime] \label{example:sparse}
Figure \ref{fig1} shows a realization of a random reaction network $G_n$ generated by the type-homogeneous stochastic block model with $n=8$ and $p_n =\frac{0.5}{n^{3.5}}$ (which is in the sparse regime). 
The following properties of $G_n$ are as expected from Theorem~\ref{thm:maintheorem}: 
The deficiency is $
\delta = 10 - 3 - 7 =0 
$ 
and so $G_n$ is not multistationary, and  it is easy to check that $G_n$ has unconditional ACR in all species except $X_3$ (which does not appear in any complex).
\end{example}

\begin{figure}[htb]
\includegraphics[scale=0.8]{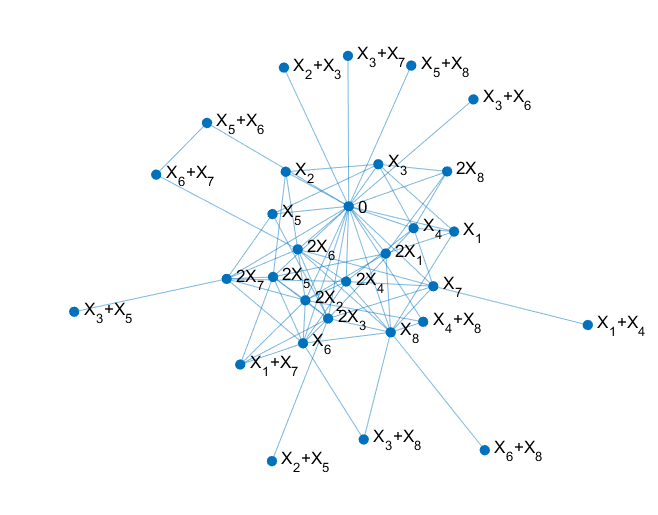}
\caption{A realization of a random reaction network generated by the type-homogeneous stochastic block model in the dense regime ($n=8$ and $p_n=\frac{2.5}{n^{3}}$). } 
\label{fig2}
\end{figure}

\begin{figure}[htb]
\includegraphics[width=10cm,height=5cm]{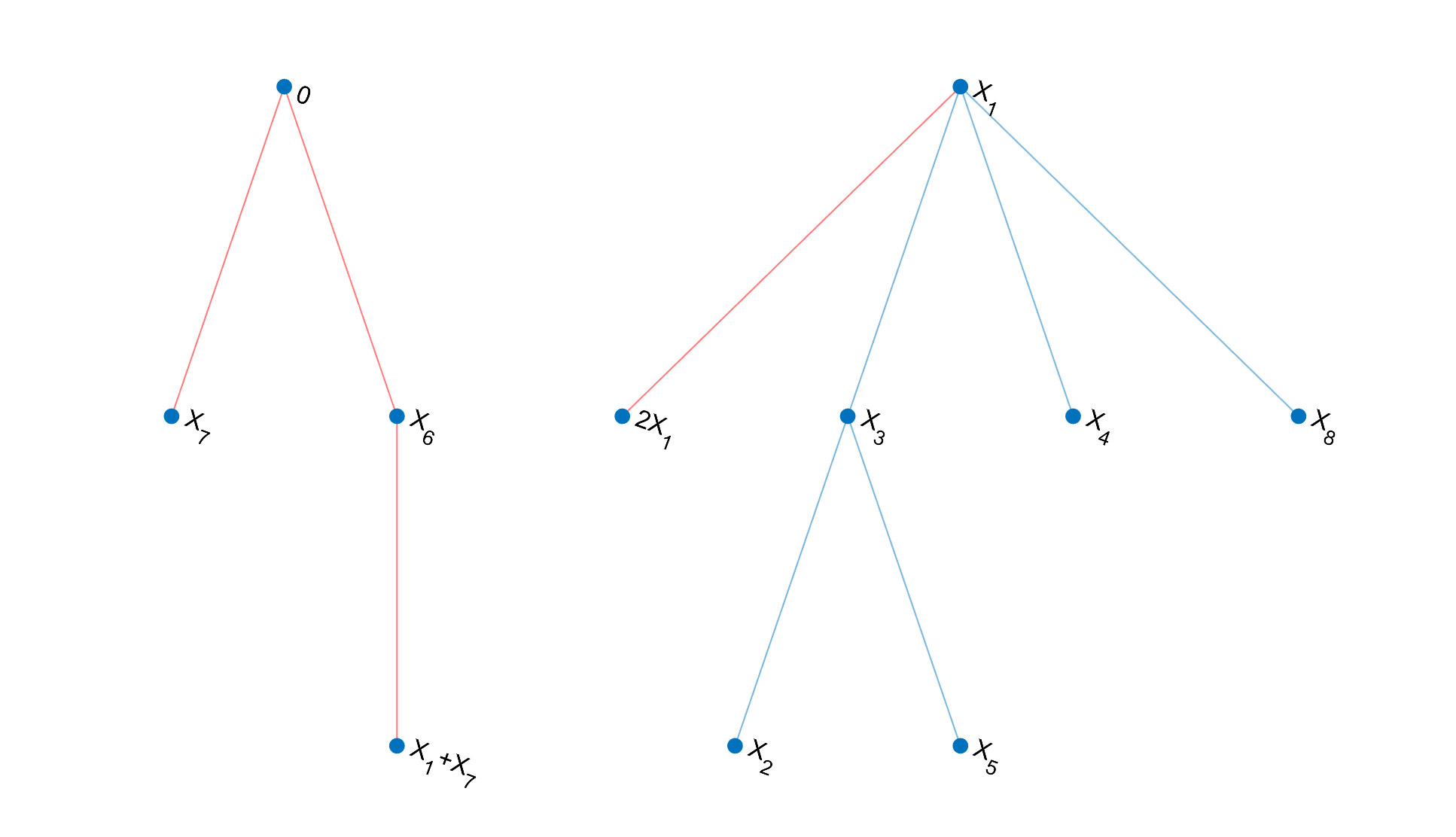}
\caption{A subnetwork of the network in Figure \ref{fig2}, which is a union of a multistationary motif $M$ (red edges) and a lifting component $L$ (blue edges). 
Notice that $M$ and $L$ share exactly one  species, namely, $X_1$.}
% {\color{red} Tung, can you check this example?  All 3 inflows appear among the red reactions, but the multistationary motifs don't have that...}{\color{blue}I have erased the unneeded edge.}}
\label{fig3}
\end{figure}

% EXAMPLE STARTS HERE
\begin{example}[Dense regime]\label{example:dense}
Figure \ref{fig2} shows a realization of a random reaction network $G_n$ generated by the type-homogeneous stochastic block model with $n=8$ and $p_n=\frac{2.5}{n^{3}}$ (which is in the dense regime). 
Figure~\ref{fig3} depicts a subnetwork of $G_n$ that is a union of a multistationary motif and a lifting component with one species in common. We can now use Proposition~\ref{prop:MM+LC} and Lemma~\ref{lem:monotone} to 
assert that this subnetwork is multistationary and lacks ACR, and then 
``lift'' %nondegenerate 
these properties
  to $G_n$.
%, the network $G$ has nondegenerate multistationarity and does not have unconditional ACR.
Indeed, this approach underlies our proof of Theorem~\ref{thm:maintheorem} in the dense regime.
\end{example}

\begin{remark}[Window of co-existence] \label{remark:window}
If the number of species satisfies 
$n<e^{17/2} \approx 4914.8$, then 
the small window between $\frac{1}{n^3}$ and $\frac{\frac{2}{17}\log(n)-c(n)}{n^3}$, in Theorem~\ref{thm:maintheorem}(ii), does not exist. Therefore, in the type-homogeneous stochastic block model, it is unlikely to observe a random network with both multistationarity and ACR, unless it has many species.
\end{remark}

% --------------------------
% proof of theorem (propositions below)
% --------------------------

\subsection{Proof of Theorem~\ref{thm:maintheorem}} \label{sec:proof}

Theorem~\ref{thm:maintheorem} follows directly from Propositions \ref{prop:sparse} and
\ref{prop:mss}--\ref{prop:ACR} below. 
%\ref{prop:mss}, \ref{prop:noACR}, and \ref{prop:ACR} below. 
%The rest of this section
 This subsection is devoted to proving these propositions, which requires the following lemma. % that we make note of now.
%We first make note of an inequality that is often used in our subsequent proofs.
%The following inequality is frequently used in the proof of these propositions.

\begin{lemma} \label{lem:inequality}
For all  $n \geq 1$
%$n \in \mathbb{Z}_{>0}$ 
and $0 \leq x \leq 1$, the following inequality holds:
%Since $\log(1-x) \leq -x$ for any $x\in (0,1)$, we have 
\[
(1-x)^n ~ \leq ~ e^{-nx}~.
\]
%for any $x\in (0,1)$, $n>0$. 
\end{lemma}

\begin{proof}
If $x=1$, the inequality holds.  For $0 \leq x <1$, the result follows directly from the inequality $\log(1-x) \leq -x$ (which is easy to check) and the fact that the log function is increasing.
\end{proof}

%{\color{orange}
%In the propositions below, they all start ``Consider random reaction networks $G_n$ generated by edge probabilities given by \eqref{edge_probability}.''  Maybe we can define a Setup ($\star$) and then just say, ``Assume Setup ($\star$)''. }

%\medskip

\subsubsection{Sparse regime}
%\paragraph{\textbf{Sparse regime}.}
Our result for the sparse regime is  Proposition~\ref{prop:sparse} below.  
Its proof uses 
 Lemma~\ref{lem:def-0} % (pertaining to deficiency-zero networks)}
and recent results on the prevalence of deficiency-zero networks~\cite{prevalence_block} .
\begin{proposition}[Sparse regime]\label{prop:sparse}

Consider random reaction networks $G_n$ generated by edge probabilities given by \eqref{edge_probability}.
If $\frac{1}{n^4}\ll  p_n\ll \frac{1}{n^{10/3}}$, then, w.h.p.\ $G_n$
has deficiency zero, 
			is not multistationary, 
			and has unconditional ACR (in some species).
%When $\frac{1}{n^4}\ll  p_n\ll \frac{1}{n^{10/3}}$, w.h.p. $G_n$ is not multistationary, and $G_n$ has unconditional ACR in some species.
\end{proposition}
\begin{proof} 
Assume that $\frac{1}{n^4}\ll  p_n\ll \frac{1}{n^{10/3}}$.
It follows from \cite[Theorem 5.1 and Example 10]{prevalence_block} that, w.h.p., the deficiency of $G_n$ is 0.
Thus, w.h.p., the deficiency-zero theorem (Lemma~\ref{lem:def-0}(1)) applies and so $G_n$ is not multistationary.  
Additionally, by Lemma~\ref{lem:def-0}(2), 
to show that w.h.p. $G_n$ has unconditional ACR in some species, it suffices to show that w.h.p. $G_n$ contains an edge in $E_{0,1}$. 
The probability that any given edge in $E_{0,1}$ appears in $G_n$ is $n^{4-0-1}p_n=n^3p_n$, and there are $|E_{0,1}| = 2n$ such edges, so:
	\begin{align} \label{eq:edge-0-1-pbty}
	\P(G_n \text{ contains an edge in } E_{0,1}) ~=~ 1-\big(1-n^3p_n)^{2n} ~\geq ~ 1 - e^{-2n^4p_n}~.
	\end{align}
(The inequality in~\eqref{eq:edge-0-1-pbty} is due to Lemma~\ref{lem:inequality}.) 
Finally,  
using~\eqref{eq:edge-0-1-pbty} and the assumption $p_n\gg\frac{1}{n^4}$, 
we obtain that 
$\lim_{n\to\infty}\P(G_n \text{ contains an edge in  } E_{0,1}) = 1$.  
This concludes the proof.
\end{proof}

% EXAMPLE WAS HERE.  ANNE MOVED IT EARLIER

\subsubsection{Dense regime}
The proofs in this subsection make frequent use of the well-known second moment method (for example, see \cite{alon2016probabilistic}). 
We summarize this approach in the following lemma.
%We first state it as a lemma.
\begin{lemma}\label{lem:2nd_moment}
Let $\{T_n\}$ be a sequence of non-negative random variables. 
If $\var(T_n) \ll (\E T_n)^2$, then $\lim\limits_{n\to\infty}\P(T_n>0) = 1$.     
\end{lemma}
\begin{proof}
From the second moment method, we have
\[
\P(T_n >0) ~\geq~ 1 - \frac{\var(T_n)}{(\E T_n)^2}.
\]
Taking the limit (as $n\to\infty$) completes the proof.  
\end{proof}

Next, we show that networks generated in the dense regime  
contain multistationary motifs~\eqref{eq:multi-motif} w.h.p. (see Figures \ref{fig2} and \ref{fig3} for an example).

\begin{lemma}\label{lemma:MM}
Consider random reaction networks $G_n$ generated by edge probabilities given by \eqref{edge_probability}.
If $p_n \gg \frac{1}{n^3}$, 
then w.h.p.\
some multistationary motif in the set $S_{M,n}$ is a subnetwork of $G_n$.
%$G_n$ contains as a subnetwork, a multistationary motif in the set $S_{M,n}$, w.h.p. 
\end{lemma}

\begin{proof}
Assume $p_n \gg \frac{1}{n^3}$. 
Then $G_n$ contains all inflows/outflows $0 \leftrightarrows X_i$ (recall Remark~\ref{rem:all-flows}), so 
it suffices to show that w.h.p. $G_n$ contains a subnetwork of the following form, for some $i,j,k$ distinct:
\begin{align} \label{eq:motif-ijk}
\{X_k\leftrightarrows 2X_k, ~ X_i\leftrightarrows X_j + X_k\}~.
\end{align}

Consider the reactions in~\eqref{eq:motif-ijk}.  
First, $X_k\leftrightarrows 2X_k$ is in $ E_{1,1}$, so its edge probability is $n^2p_n\wedge 1$. 
Next, for a fixed $k$, with $1 \leq k \leq n$, there are $(n-1)(n-2)$ reactions of the form $X_i\leftrightarrows X_j + X_k$ with $i,j,k$ distinct. 
Each such reaction belongs to $E_{1,2}$ and so its edge probability is $np_n\wedge 1$.

%It is straightforward to check that 
We can reduce to considering only three cases: 
(1) when  $p_n >  \frac{1}{n}$ for all $n$, 
(2) when  $\frac{1}{n^2} \leq p_n \leq  \frac{1}{n}$ for all $n$, and 
(3) when $p_n<\frac{1}{n^2}$ for all $n$.

\noindent \textbf{Case 1:} $p_n >  \frac{1}{n}$ for all $n$.
In this case, $n^2p_n  \wedge 1 = 1=n p_n \wedge 1 $.  So, for all $n \geq 3$, $G_n$ contains a subnetwork of the form~\eqref{eq:motif-ijk} (in fact, $G_n$ contains all possible such subnetworks).

\noindent \textbf{Case 2:} $\frac{1}{n^2} \leq p_n \leq  \frac{1}{n}$ for all $n$.
In this case, $n^2p_n \wedge 1=1$, so
$G_n$ contains 
 all reactions  of the form $X_k\leftrightarrows 2X_k$. 
 Hence, we need only show that w.h.p.\ 
 $G_n$ contains at least one reaction of the form  $X_i\leftrightarrows X_j + X_k$ with $i,j,k$ distinct.
 The probability of this event, which we call $E_n$, is as follows:
\[
\P(E_n) ~=~
1-(1-np_n)^{n(n-1)(n-2)} ~\geq~ 1 - e^{-n^2(n-1)(n-2)p_n} ~ \to ~ 1 \quad  \mathrm{(as }~ n \to \infty \mathrm{)},
\]
where the inequality is due to Lemma~\ref{lem:inequality}, 
and the limit comes from the fact that  $p_n\geq \frac{1}{n^2}$.  %This completes Case~2.
%Since $p_n\geq \frac{1}{n^2}$ we have $\lim_{n\to\infty}1 - e^{-n^2(n-1)(n-2)p_n}=1$, thus $G_n$ contains a reaction of the form $X_i\leftrightarrows X_j + X_k$ w.h.p.

\noindent \textbf{Case 3:}  
 $p_n<\frac{1}{n^2}$ for all $n$.
In this case, the edge probability for each reaction $X_k\leftrightarrows 2X_k$ (respectively, $X_i\leftrightarrows X_j+X_k$) is
 $n^2p_n$ (respectively, $np_n$).

For $1 \leq k \leq n$, 
let $A_{k}$ denote the event that $G_n$ contains a subnetwork of the form~\eqref{eq:motif-ijk}, where $i\neq j$ and $i,j \neq k$.  
 (The notation $A_{k,n}$ would be better for $A_k$, but we prefer to avoid excessive subscripts.)
It follows that the probability of $A_k$ is:
\begin{align} \label{eq:pbty-A}
\P(A_k) ~=~ n^2p_n\left( 1-(1-np_n)^{(n-1)(n-2)} \right)~.
\end{align}

Define the random variable $T_n :=  \sum_{k=1}^n {1}_{A_k}$.   
We wish to show that
$\lim_{n\to\infty} \P(T_n>0) = 1$. 
By Lemma \ref{lem:2nd_moment}, it is enough to prove $\var(T_n)\ll (\E T_n)^2$. 
% \begin{equation}\label{2ndmoment}
% \lim_{n\to\infty}\frac{\var(T_n)}{(\E T_n)^2} ~=~ 0~.
% \end{equation}
To this end, we first compute $\E T_n$, using~\eqref{eq:pbty-A}:
%\begin{small}
\begin{align} \label{eq:ExpectationT}
%\P(A_k) = n^2p_n(1-(1-np_n)^{(n-1)(n-2)})
 %\text{\quad and \quad} 
 \E T_n ~=~ \sum_{k=1}^n \P(A_k) ~=~ n^3p_n \left( 1-(1-np_n)^{(n-1)(n-2)} \right)~.
\end{align}
%\end{small}

Next, Lemma~\ref{lem:inequality} yields the first inequality here:
	\begin{align} \label{eq:inequality-less-than-1}
	(1-np_n)^{(n-1)(n-2)}
	~\leq~ e^{-n(n-1)(n-2)p_n}
	~ \ll~ 1~,
	\end{align}
and the second inequality (limit) comes from the assumption that $p_n\gg\frac{1}{n^3}$.
%, we get, for $n$ sufficiently large, $(1-np_n)^{(n-1)(n-2)}\leq e^{-n(n-1)(n-2)p_n} \ll 1$. 
Hence, using~\eqref{eq:ExpectationT}, we obtain $\E T_n\sim n^3p_n$. % \gg 1$. 

%{\color{brown} \sout{
%Next, we show that  the events $A_k$  are ``almost" pairwise independent, which implies the sum $T_n$ concentrates around the mean. } {\color{red} [or reword so it is clear this comment is for readers already familiar with this approach?]}}
%
To compute $\var(T_n)$, we consider the event $A_h\cap A_k$, where $h\neq k$.  It is straightforward to check that  $A_h\cap A_k$ occurs if and only if $G_n$ contains the reactions $X_h\leftrightarrows 2X_h$ and  $X_k\leftrightarrows 2X_k$ and also one of the following:
	\begin{enumerate}
		\item a reaction of the form $X_i\leftrightarrows X_h+X_k$ (for some $i \neq h,k$), or 
		\item a reaction of the form $X_i\leftrightarrows X_j + X_k$ (for some $j\neq k, h$ and $i\neq j, k$) and a reaction of the form $X_l \leftrightarrows X_m + X_h$ (for some $m\neq h, k$ and $l\neq m,h$).
	\end{enumerate}

%{\color{blue} \sout{
%In this event, both $X_k\leftrightarrows 2X_k$ and $X_h\leftrightarrows 2X_h$ are present. 
%In addition, either 
%$X_i\leftrightarrows X_h+X_k$ is present, or 
%there is a reaction of the form $X_i\leftrightarrows X_j + X_k$ with $j\neq k, h$ and $i\neq j, k$, 
%and a reaction of the form $X_l \leftrightarrows X_m + X_h$ with $m\neq h, k$ and $l\neq m,h$. This gives:}}
A direct computation now yields the following probability:
\begin{align} \label{eq:pbty-intersection}
	\P(A_h\cap A_k)  \notag
%	&=
%	 (n^2p_n)^2 
%	 \left(
%	 1 - (1-n p_n) ^{n-2} \left( 1 - (1 - (1-n p_n)^{{(n-2)^2}} )^2 \right)
%	 \right) \\
	&~=~ (n^2p_n)^2 
			(1-(1-np_n)^{n-2} + (1-np_n)^{n-2} (1-(1-np_n)^{(n-2)^2})^2)\\
%	&=n^4p_n^2(1-2(1-np_n)^{(n-2)^2+n-2}+(1-np_n)^{2(n-2)^2+n-2})\\
        &~=~ n^4p_n^2(1-2(1-np_n)^{(n-1)(n-2)}+(1-np_n)^{(n-2)(2n-3)})~.
\end{align}
Now we use equations~\eqref{eq:pbty-A}, \eqref{eq:ExpectationT}, and \eqref{eq:pbty-intersection}
to 
 compute $\var(T_n)$, as follows:

\begingroup
\allowdisplaybreaks
\begin{small}
\begin{align} \label{eq:variance}
	\notag
    \var(T_n) &~=~ \E(T_n^2) - (\E T_n)^2 \\
    	\notag
    & ~=~ 
    		\sum_{k=1}^n\P(A_k)+\sum_{h\neq k}\P(A_h\cap A_k) 
    	- (\E T_n)^2 \\
	%- n^6p_n^2 (1-(1-np_n)^{(n-1)(n-2)})^2\\
	\notag
    & ~=~ n^3p_n \left( 1-(1-np_n)^{(n-1)(n-2)} \right) + (n-1)n^5p_n^2\left( 1-2(1-np_n)^{(n-1)(n-2)} + (1-np_n)^{(n-2)(2n-3)} \right) \\
        	& \quad \quad - n^6p_n^2 \left(1-(1-np_n)^{(n-1)(n-2)} \right)^2 ~. %\\
	%- n^6p_n^2 (1-(1-np_n)^{(n-1)(n-2)})^2\\
%    &~=~ n^3p_n(1-(1-np_n)^{(n-1)(n-2)}) - n^5p_n^2(1-2(1-np_n)^{(n-1)(n-2)}+(1-np_n)^{(n-2)(2n-3)}) \\
%    	& \quad \quad  + n^6p_n^2(2(1-np_n)^{(n-1)(n-2)}-2(1-np_n)^{(n-1)(n-2)}+(1-np_n)^{(n-2)(2n-3)}-(1-np_n)^{2(n-1)(n-2)})\\
%    &~=~ n^3p_n(1-(1-np_n)^{(n-1)(n-2)}) - n^5p_n^2(1-2(1-np_n)^{(n-1)(n-2)}+(1-np_n)^{(n-2)(2n-3)}) \\
%    	& \quad \quad  + n^6p_n^2(1-np_n)^{(n-2)(2n-3)}(1-(1-np_n)^{n-2}).
\end{align}
\end{small}
\endgroup

We claim that $\var(T_n) \ll n^6p_n^2$.  Indeed, this follows in a straightforward way 
from~\eqref{eq:inequality-less-than-1} and~\eqref{eq:variance}, 
 the limit $(1-np_n)^{(n-2)(2n-3)} \ll 1$ (which is closely related to~\eqref{eq:inequality-less-than-1}), and the assumption $p_n\gg\frac{1}{n^3}$.  
%Since $p_n\gg\frac{1}{n^3}$ and $(1-np_n)^{(n-2)(2n-3)} \leq e^{-np_n((n-2)(2n-3))} \ll 1$, we have $\var(T_n) \ll n^6p_n^2$. 
Finally, having shown $\var(T_n) \ll n^6p_n^2$ and $\E T_n \sim n^3p_n$, we get, as desired,  $\var(T_n)\ll (\E T_n)^2$.
%the desired inequality (limit): $\var(T_n)\ll (\E T_n)^2$.
%
%in \eqref{2ndmoment}.
% \begin{align*}
% 	\P(A_h\cap A_k) 
% 	&= (n^2p_n)^2 
% 		(np_n + {(1-np_n)} (1-(1-np_n)^{(n-2)^2})^2)\\
% 	&=n^4p_n^2(1-2(1-np_n)^{(n-2)^2+1}+(1-np_n)^{2(n-2)^2+1}).
% \end{align*}
% Now we can compute $\var(T_n)$
% \begingroup
% \allowdisplaybreaks
% \begin{small}
% \begin{align*}
%     \var(T_n) &= \E(T_n^2) - (\E T_n)^2 \\
%     & = \sum_{k=1}^n\P(A_k)+\sum_{h\neq k}\P(A_h\cap A_k) - n^6p_n^2 (1-(1-np_n)^{(n-1)(n-2)})^2\\
%     & = n^3p_n(1-(1-np_n)^{(n-1)(n-2)}) + (n-1)n^5p_n^2(1-2(1-np_n)^{(n-2)^2+1}+(1-np_n)^{2(n-2)^2+1}) \\
%     &- n^6p_n^2 (1-(1-np_n)^{(n-1)(n-2)})^2\\
%     &= n^3p_n(1-(1-np_n)^{(n-1)(n-2)}) - n^5p_n^2(1-2(1-np_n)^{(n-2)^2+1}+(1-np_n)^{2(n-2)^2+1}) \\
%     & + n^6p_n^2(2(1-np_n)^{(n-1)(n-2)}-2(1-np_n)^{(n-2)^2+1}+(1-np_n)^{2(n-2)^2+1}-(1-np_n)^{2(n-1)(n-2)})\\
%     &=n^3p_n(1-(1-np_n)^{(n-1)(n-2)}) - n^5p_n^2(1-2(1-np_n)^{(n-2)^2+1}+(1-np_n)^{2(n-2)^2+1}) \\
%     & + n^6p_n^2(1-np_n)^{(n-2)^2+1}(2(1-np_n)^{n-3}-2+(1-np_n)^{(n-2)^2}-(1-np_n)^{n(n-2)-1}).
% \end{align*}
% \end{small}
% \endgroup
% Since $p_n\gg\frac{1}{n^3}$ and $(1-np_n)^{(n-2)^2+1} \leq e^{-np_n((n-2)^2+1)} \ll 1$, we have $\var(T_n) \ll n^6p_n^2$. Combined with $\E T_n \sim n^3p_n$, we obtain the desired limit in \eqref{2ndmoment}.
\end{proof}

%{\color{brown} The conversation Anne/Nidhi/Tung had about the end of the proof is commented out, but easily seen in ``old version'' PDFs.} %we had, just moved down here.}

%{\color{teal}
%\begin{align*}
%	\P(A_h\cap A_k) 
%	&=
%	 (n^2p_n)^2 
%	 \left(
%	 1 - (1-n p_n) ^{n-2} \left( 1 - (1 - (1-n p_n)^{{(n-2)^2}} )^2 \right)
%	 \right)
%\end{align*} }
%{\color{red} Anne got the expression above.  Original is below.  Tung or Nidhi, can you check?}
%
%\blue{[N.]I don't completely follow how we get to teal expression but I do think in the original expression some powers are missing and now I am getting this:
%\begin{align*}
%\P(A_h\cap A_k) 
%&= (n^2p_n)^2 
%((np_n)^{n-2} + {(1-np_n)^{n-2}} (1-(1-np_n)^{(n-2)^2})^2).
%\end{align*}}
%{\color{blue} [Tung: I agreed with Anne. In the original version I fixed i somehow. In Nidhi's expression the probability of having $X_i\leftrightarrows X_h+X_k$ should be $1-(1-np_n)^{n-2}$ (and not $(np_n)^{n-2}$), which makes the whole expression the same as Anne's. I fixed the rest below, please double check.]
%\begin{align*}
%	\P(A_h\cap A_k) 
%	&= (n^2p_n)^2 
%		(1-(1-np_n)^{n-2} + (1-np_n)^{n-2} (1-(1-np_n)^{(n-2)^2})^2)\\
%\end{align*}
%
%}

Lemma \ref{lemma:MM} allows us to establish the threshold for nondegenerate multistationarity, as follows:

\begin{proposition}[Multistationarity in dense regime]\label{prop:mss}

Consider random reaction networks $G_n$ generated by edge probabilities in \eqref{edge_probability}.
If $p_n \gg \frac{1}{n^3}$, then $G_n$ is nondegenerately multistationary~w.h.p. 
\end{proposition}
\begin{proof}
Assume $p_n \gg \frac{1}{n^3}$.
By Lemma~\ref{lemma:MM}, w.h.p., $G_n$
contains (as a subnetwork) a multistationary motif $M \in S_{M,n}$ (which is $3$-dimensional and nondegenerately multistationary, as noted in Remark~\ref{rem:motif-has-mss}).
Relabeling species, if needed, we may assume that the species of $M$ are $X_1,X_2,X_3$.

Next, $p_n\gg \frac{1}{n^3}$ implies that, w.h.p., 
$G_n$ contains all inflow/outflow reactions $0 \leftrightarrows X_i$ (Remark~\ref{rem:all-flows}), 
and in particular contains the $(n-3)$-dimensional subnetwork consisting of reactions $0 \leftrightarrows X_i$, for all $i=4,5, \dots, n$, which we denote by $G'$.  
%$G'=\cup_{i=4}^n 0\leftrightarrows X_i$. Since 
As $M$ and $G'$ have no species in common,
the subnetwork of $G$ formed by the union of their reactions, which we denote by $N$, is full-dimensional ($n$-dimensional).
 %the subgraph $M\cup G'$ of $G_n$ is full-dimensional 
It is straightforward to check that $N$ ``inherits'' nondegenerate multistationarity from $M$. (The proof is similar to that of 
Proposition~\ref{prop:MM+LC}(1).) 
 Thus, by Lemma \ref{lem:monotone},  
$G_n$ is nondegenerately multistationary  w.h.p. 
\end{proof}

\begin{proposition}[ACR in dense regime]\label{prop:noACR}
Consider random reaction networks $G_n$ generated by edge probabilities given by \eqref{edge_probability}.
 If the following inequality holds:
 \begin{align} \label{eq:inequality-dense}
 \ds p_n ~ \geq~ \frac{\log(n-2)+c(n)}{n^2(n-2)}~, \quad \textrm{ for some } c(n)\to\infty~,
 \end{align}
 then w.h.p. $G_n$ 
 %\sout{is nondegenerately multistationary and} 
 does {\bf not} have  unconditional ACR (in any species).
 \end{proposition}

\begin{proof}
% {\color{brown} Combining Proposition~\ref{prop:MM+LC}, Lemmas~\ref{lemma:MM}, \ref{theorem:LC} and the fact that $ \frac{\log(n-2)+c(n)}{n^2(n-2)} \gg \frac{1}{n^3}$ we obtain the following result.}. [ADD MORE DETAIL, BELOW?]
% {\color{blue} Outline:
% \begin{itemize}	
% 	\item $ \frac{\log(n-2)+c(n)}{n^2(n-2)} \gg \frac{1}{n^3}$, so by  Lemma~\ref{lemma:MM} we get multistationary motifs w.h.p.
% 	\item Use those motifs to define corresponding ``one-shared-species-only sets'' $\Sigma_n$?  {\color{red} Not sure.}
% 	\item Use Lemma \ref{theorem:LC} (and Proposition~\ref{lem:monotone}?) to show such lifting components (with vertex set $\Sigma_n$) exist.
% 	\item Combine the two to get joined subnetworks $H_n \in S_{J,n}$.
% 	\item Use $H_n$ and Proposition~\ref{lem:monotone} to lift to $G_n$.
% 	\end{itemize}
% }
Assume that inequality~\eqref{eq:inequality-dense} holds.  
By Proposition~\ref{prop:MM+LC}(2), it suffices to show that, w.h.p., some $H \in S_{J,n}$ (as in Definition~\ref{def:joined-sets}) is a subnetwork of $G_n$.

First, consider the case when $p_n \geq \frac{1}{n^2}$ for all $n$.
Then, $p_n \gg \frac{1}{n^3}$, so by Lemma~\ref{lemma:MM}, $G_n$ contains, as a subnetwork, some multistationary motif $M \in S_{M,n}$.  
Next, we show that $G_n$ also contains all lifting components involving species $\{X_1, X_2, \dots, X_n\}$.  Indeed, reactions of the form $X_{\ell} \leftrightarrows X_m$ are in $E_{1,1}$ and hence have edge-probability $n^{4-1-1} p_n \wedge 1 = 1$ (since $p_n \geq \frac{1}{n^2}$), and so $G_n$ contains all such reactions.  
Thus, as desired, w.h.p., $G_n$ contains some $H \in S_{J,n}$ as a subnetwork.

%{\color{blue} \sout{
%In this case, $G_n$ contains all reactions of the form $X_k\leftrightarrows 2X_k$ and $X_i\leftrightarrows X_j$. 
%Furthermore, from the proof of Lemma \ref{lemma:MM}, w.h.p. $G_n$ contains a reaction of the form $X_i\leftrightarrows X_j+X_k$ with $i,j,k$ distinct. Thus w.h.p. $G_n$ contains a subnetwork in $S_{J,n}$, which implies it does not have unconditional ACR in any species by Lemma~\ref{lem:monotone} and Proposition~\ref{prop:MM+LC}. }}

To complete the proof, we need only consider the following case:
%For the rest of the proof, it is enough we can assume the following: 
	\begin{align} \label{eq:inequality-rest-of-proof}
	\ds \frac{\log(n-2)+c(n)}{n^2(n-2)}  ~ \leq ~ p_n ~<~ \frac{1}{n^2}~, \quad \textrm{ for some } c(n)\to\infty~.
	\end{align}

%{\color{blue} If $p_n \geq \frac{1}{n^2}$, $G_n$ must contain all reactions of the form $X_k\leftrightarrows 2X_k$ and $X_i\leftrightarrows X_j$. Furthermore, from the proof of Lemma \ref{lemma:MM}, w.h.p. $G_n$ contains a reaction of the form $X_i\leftrightarrows X_j+X_k$ with $i,j,k$ distinct. Thus w.h.p. $G_n$ contains a subnetwork in $S_{J,n}$, which implies it does not have unconditional ACR in any species by Propositions~\ref{lem:monotone} and~\ref{prop:MM+LC}. For the rest of the proof, we can assume $\ds \frac{\log(n-2)+c(n)}{n^2(n-2)} \leq p_n <\frac{1}{n^2}$ for some $c(n)\to\infty$.}

Recall from Remark~\ref{rem:all-flows} that, in this case, $G_n$ contains all reactions of the form $0\leftrightarrows X_{\ell}$. Thus, all vertices of the form $X_{\ell}$ appear in $G_n$.
For positive integers $i\neq j$, 
let $G_n^{i,j}$ denote the subgraph of (the underlying graph of) $G_n$, induced by the following set of vertices of $G_n$: 
$\left\{X_{\ell} \mid \ell \in \{1,2,\dots, n\} \smallsetminus \{i,j\} \right\}$. 
%1\leq m\leq n, m\neq i, m\neq j\}$. 
Next, for distinct $i,j,k \in \{1,2,\dots, n\}$, 
%$1\leq k,i,j \leq n$ distinct, 
let $A_{k,i,j}$ denotes the event that 
(i)~$G_n^{i,j}$ is connected and
(ii) $G_n$ contains the reactions $X_k\leftrightarrows 2X_k$ and $X_i\leftrightarrows X_j+X_k$. 

We claim that, for $n$ sufficiently large, the event $A_{k,i,j}$ implies that $G_n$ contains some $H \in S_{J,n}$ as a subnetwork.  
To see this, first note that the inequality~\eqref{eq:inequality-rest-of-proof} implies that $p_n \gg \frac{1}{n^3}$ and so, for $n$ large enough, $G_n$ contains all flow reactions $0 \leftrightarrows X_{\ell}$
(Remark~\ref{rem:all-flows}).   So, condition~(ii) guarantees a multistationary motif $M$, for $n$ sufficiently large.  Next, condition~(i) and the fact that connected graphs have spanning trees yield a ``complementary'' lifting component $L$. 
By joining $M$ and $L$, we obtain some $H \in S_{J,n}$  as a subnetwork of $G_n$, as claimed.
%Thus, the event $A_{k,i,j}$ implies that $G_n$ contains some $H \in S_{J,n}$ obtained by joining $M$ and $L$.

%\sout{
%By Lemma~\ref{lem:monotone} and Proposition~\ref{prop:MM+LC}, the  event $\bigcup_{i,j, k}A_{k,i,j}$ 
%(where the union is over distinct $i,j,k \in \{1,2,\dots, n\}$) 
%implies that $G_n$ is nondegenerately multistationary and does not have unconditional ACR in any species.} 

Let $T_n=\sum_{k,i,j}1_{A_{k,i,j}}$ 
(where the sum is over distinct $i,j,k \in \{1,2,\dots, n\}$). 
To finish the proof, it suffices to show that 
%\[
$\lim_{n\to\infty} \P(T_n>0) =1.$
%\]
% As in the proof of Lemma \ref{lemma:MM}, we prove this limit by the second moment method. In particular, we need to show the following:
% \[
% \lim_{n\to\infty} \frac{\var(T_n)}{(\E T_n)^2} = 0.
% \]
By Lemma \ref{lem:2nd_moment}, we need only show $\var(T_n)\ll (\E T_n)^2$.

Again, we start by computing $\E T_n$. Each edge of $G_n^{i,j}$ belongs to $E_{1,1}$, and so its edge probability~\eqref{edge_probability} is 
	$(n^2p_n \wedge 1)  = n^2 p_n \geq   \frac{\log(n-2)+c(n)}{n-2} $ (here we use~\eqref{eq:inequality-rest-of-proof}). 
% {\color{red}[Tung: we don't need $\wedge 1$ here since we already said earlier that for the rest of the proof we can assume $p_n\ll\frac{1}{n^2}$.].} {\color{teal} FIXED.}
 It is well known that $\frac{\log n}{n}$ is the edge-probability threshold for connectivity of random graphs with $n$ vertices and uniform edge probabilities~\cite{ER:giantcomponent}. So, for any $i\neq j$ (with $1 \leq i, j \leq n$), we have:
\begin{align} \label{eq:conn}
\lim_{n\to\infty}\P(G_n^{i,j} \text{ is connected} )~=~ 1~.
\end{align}
We emphasize that the above probability does not depend on the choice of $i,j$. So, for convenience, we denote $d_n:=\P(G_n^{i,j} \text{ is connected} )$. 
 Next, we compute the following probability using~\eqref{edge_probability}:
\[
\P(A_{k,i,j}) ~=~ (n^2p_n) (np_n) d_n ~=~ n^3p_n^2 d_n~,
\]
which implies the following:
\begin{align} \label{eq:expectation-Tn}
\E T_n ~=~ \sum_{k,i,j}\P(A_{k,i,j}) ~=~ n(n-1)(n-2) n^3p_n^2 d_n
	~=~
	n^4 (n-1)(n-2)  p_n^2 d_n 
	 ~.
%	 ~ \leq ~ n^6 p_n^2 d_n~.
\end{align}
Recall that $d_n\sim 1$ (from~\eqref{eq:conn}), so we have $\E T_n \sim n^6p_n^2$. 

Next, we analyze $\var(T_n)$ by first computing $\P(A_{k_1,i_1,j_1}\cap A_{k_2,i_2,j_2})$ where $(k_1,i_1,j_1)\neq (k_2,i_2,j_2)$. We have the following cases. 

% CASE 1
\noindent \textbf{Case 1:} $k_1=k_2$. In this case, we have $(i_1,j_1)\neq(i_2,j_2)$, and 
 the number of such pairs of events is $\leq n^5$.  
%\sout{Thus there are $n(n-1)(n-2)((n-1)(n-2) - 1)$ such pairs of events.}
Each pair of events occurs precisely when $G_n$ contains the (distinct) reactions $X_{k_1}\leftrightarrows 2X_{k_1}$, $X_{i_1}\leftrightarrows X_{j_1}+X_{k_1}$, $X_{i_2}\leftrightarrows X_{j_2}+X_{k_1}$, and both $G_n^{i_1,j_1}$ and $G_n^{i_2,j_2}$ are connected. Thus, for this case, we use~\eqref{edge_probability} to compute:
\[
\P(A_{k_1,i_1,j_1}\cap A_{k_2,i_2,j_2}) ~\leq~ n^2p_n(np_n)^2d_n ~=~ n^4p_n^3d_n.
\]
% CASE 2
\noindent \textbf{Case 2:} $k_1\neq k_2, i_1=i_2,j_1=k_2,j_2=k_1$. 
The number of such pairs of events is $n(n-1)(n-2)$. % $\leq n^3$.  
Each pair of events occurs when $G_n$ contains the (distinct) reactions $X_{k_1}\leftrightarrows 2X_{k_1}$, $X_{k_2}\leftrightarrows 2X_{k_2}$, $X_{i_1}\leftrightarrows X_{j_1}+X_{k_1}$, and both $G_n^{i_1,j_1}$ and $G_n^{i_2,j_2}$ are connected. Thus for this case 
\[
	\P(A_{k_1,i_1,j_1}\cap A_{k_2,i_2,j_2}) ~\leq~ (n^2p_n)^2(np_n)d_n 
	~=~ n^5p_n^3d_n.
\]
% CASE 3
\noindent \textbf{Case 3:} $k_1\neq k_2$ and either $i_1=i_2, (j_1,j_2)\neq(k_2,k_1)$ or $i_1\neq i_2$. 
 We claim that the number of such pairs of events is 
	$n(n-1)(n-2) \left( (n-1)^2(n-2) - 1 \right)$.  
	%$n(n-1)(n-2)(n^3 - 4 n^2 +5n - 3)$.  
Indeed, this number is obtained by taking the total number of pairs in Cases 2 and 3 (i.e., all pairs where $k_1 \neq k_2$) -- which is 
readily seen to be
 $n(n-1)^3(n-2)^2$ 
--
and then subtracting the number in Case 2 (and simplifying).
%$n(n-1)((n-1)^2(n-2)^2-(n-2))$.} 
%
%{\color{red} Anne got a possibly different number: first subcase yields $n(n-1)(n-2)(n-3)^2$, the second subcase yields $n(n-1)(n-2)^2\left( n-1 + (n-2)^2 \right)$.  Double-check?}
%
%{\color{blue}
%\begin{verbatim}
% For Case 3, I did k_1\neq k_2 and subtracted Case 2. 
%-For k_1\neq k_2 , we have n options for k_1 . Then we have n-1 for k_2 . 
%Since we already pick k_1 ,
%there are n-1 options for i_1 and n-2 options for j_1 (they could be equal k_2 ). 
%And since we already
%pick k_2 , there are n-1 options for i_2 and n-2 options for j_2 .
%So there are n(n-1)(n-1)^2(n-2)^2 pairs. 
%-For case 2 there are n-2 options for i_1=i_2 , and we don't need to pick j_1,j_2 
%so there are
%n(n-1)(n-2) pairs.
%Then I took the difference.
%
%In your calculation for the subcase i_1\neq i_2 
%maybe you took i_1 to be different from both k_1 and
%k_2 ? I think it just needs to be different from k_1 and 
%could still be equal to k_2 .
%\end{verbatim}
%}
%
Next, each pair of events in Case~3 occurs when $G_n$ contains the (distinct) reactions $X_{k_1}\leftrightarrows 2X_{k_1}$, $X_{k_2}\leftrightarrows 2X_{k_2}$, $X_{i_1}\leftrightarrows X_{j_1}+X_{k_1}$, $X_{i_2}\leftrightarrows X_{j_2}+X_{k_2}$, and both $G_n^{i_1,j_1}$ and $G_n^{i_2,j_2}$ are connected. Thus for this case 
\[
\P(A_{k_1,i_1,j_1}\cap A_{k_2,i_2,j_2}) ~\leq~ (n^2p_n)^2(np_n)^2d_n ~=~ n^6p_n^4d_n.
\]
% PUT CASES TOGETHER
Next, we use equation~\eqref{eq:expectation-Tn} and the analysis in Cases~1--3 to bound $\var(T_n)$:
\begin{align} \label{eq:var-bound}
	\notag
\var(T_n) &~=~ \E(T_n^2) - (\E T_n)^2\\
	\notag
	&~=~ 
	\sum_{k,i,j}\P(A_{k,i,j}) ~+ \sum_{(k_1,i_1,j_1)\neq(k_2,i_2,j_2)}\P(A_{k_1,i_1,j_1}\cap A_{k_2,i_2,j_2}) ~-~ (\E T_n)^2\\
	\notag
	&~\leq~ 
	(n^3) n^3p_n^2 d_n 
	+ 
	(n^5) n^4p_n^3d_n 
	+ 
	(n^3) n^5p_n^3d_n \\
	\notag
	& \quad  \quad +
	 %{\color{blue} n(n-1)((n-1)^2(n-2)^2-(n-2)) n^6p_n^4d_n}
	 n(n-1)(n-2) \left( (n-1)^2(n-2) - 1 \right)
	 n^6p_n^4d_n	 
	%\\ 
%	 {\color{violet} 	
%	n(n-1)(n-2)(n^3 - 4 n^2 +5n - 3)
%	 n^6p_n^4d_n} 
	 ~-~ 
	%n^2(n-1)^2(n-2)^2 n^{6} p_n^4d_n^2
	n^8(n-1)^2(n-2)^2  p_n^4d_n^2\\
	\notag
	&~ \leq ~ 
	n^6p_n^2 d_n + 
	n^9p_n^3d_n 
	+ n^8p_n^3d_n 
	+  n^8(n-1)^2(n-2)^2 p_n^4d_n 
	- 
	%n^2(n-1)^2(n-2)^2 n^{6} p_n^4d_n^2
	n^8(n-1)^2(n-2)^2 p_n^4 d_n^2\\
	&~ \leq ~ 
	(n^3p_n)^2 d_n + 
	(n^3p_n)^3d_n 
	+ n^8p_n^3d_n 
	+ (n^3p_n)^4 d_n (1-d_n)~.	
\end{align}
%{\color{brown} Anne to Tung: The term marked in orange, above, is the one I don't know how to handle.  I am guessing that this is where the $\log$ bound is used.  Can you fill in some details?  {\color{blue}[I added more explanation in the paragraph below.]} The old computation/bound for the variance is here:
%\begin{align*}
%	& ~=~ n(n-1)(n-2) n^3p_n^2 d_n + n(n-1)(n-2)((n-1)(n-2)-1)n^4p_n^3d_n \\
%	&+ n(n-1)(n-2)n^5p_n^3d_n + n(n-1)((n-1)^2(n-2)^2-(n-2))n^6p_n^4d_n \\
%	&- n^2(n-1)^2(n-2)^2n^6p_n^4d_n^2\\
%	&\leq n(n-1)(n-2) n^3p_n^2 d_n + n(n-1)(n-2)((n-1)(n-2)-1)n^4p_n^3d_n \\
%	&+ n(n-1)(n-2)n^5p_n^3d_n + n^2(n-1)^2(n-2)^2n^6p_n^4d_n(1-d_n).
%\end{align*}
%}

As noted earlier, inequality~\eqref{eq:inequality-rest-of-proof} implies that $p_n \gg \frac{1}{n^3}$.  
%Since $p_n\geq \frac{\log(n-2)+c(n)}{n^2(n-2)} \gg \frac{1}{n^3}$, 
Hence, certain terms appearing in~\eqref{eq:var-bound} have the following asymptotic properties: 
$(n^3p_n)^2 \ll  (n^3p_n)^4$ and  $n^8p_n^3 \ll (n^3p_n)^3 \ll  (n^3p_n)^4$. 
Recall also that $d_n \sim 1$.  Thus, from~\eqref{eq:var-bound}, we have 
%Furthermore, since $1-d_n \ll 1$, it is straightforward to see that 
$\var(T_n)\ll (n^3p_n)^4 = n^{12}p_n^4$. 
Finally, we showed earlier that $\E T_n \sim n^6p_n^2$, so we have $\var(T_n) \ll (\E T_n)^2$, which concludes the proof.
\end{proof}

%EXAMPLE WAS HERE.  ANNE MOVED IT EARLIER

Finally, we address the small ``window'' between the thresholds in Propositions \ref{prop:mss}--\ref{prop:noACR}. We showed that a random network in this window is nondegenerately multistationary w.h.p, and next we show that it also has ACR w.h.p.

\begin{proposition}[ACR in window of dense regime] \label{prop:ACR}
Consider random reaction networks $G_n$ generated by edge probabilities given by \eqref{edge_probability}.
 If $\{p_n\}_{n \geq 1}$ satisfies the following:
\begin{align} \label{eq:window}
 \ds \frac{1}{n^3} ~\ll~ p_n ~\leq~ \frac{\frac{2}{17}\log(n)-c(n)}{n^3} \quad \mathrm{for~some}~c(n)\to\infty~, 
 \end{align}
 then w.h.p $G_n$ has unconditional ACR in some species.
\end{proposition}

\begin{proof}
Assume~\eqref{eq:window}.  
As $\frac{1}{n^3} \ll p_n$, the random network $G_n$ contains all reactions in $E_{0,1}$, namely, $0 \leftrightarrows X_k$ and $0 \leftrightarrows 2X_k$, for all $1 \leq k \leq n$ (Remark~\ref{rem:all-flows}).
For $1 \leq k \leq n$, let $B_k= B_{k,n}$ denote the event that, in all other reactions of $G_n$, the species $X_k$ appears as a catalyst-only species.  
We claim that the event $B_k$ implies that $G_n$ has unconditional ACR in $X_k$.  Indeed, in this event, 
the mass-action ODE for species $X_k$ (for any choice of positive rate constants) has the following form:
	\[
	\frac{dx_k}{dt} ~=~ c_2x_k^2 + c_1 x_k +c_0 ~,
	\]
where $c_0 >0$ and $c_2<0$ (and $c_1 \in \mathbb{R}$). 
This quadratic polynomial has a unique positive root (by Descartes' rule).  Additionally, $G_n$ necessarily admits a positive steady state (Remark~\ref{rem:reversible}).  
We conclude that $G_n$ has unconditional ACR in $X_k$ when $B_k$ occurs.

It therefore suffices to show that, for the random variables $T_n:=\sum_{k=1}^n1_{B_k}$, 
the following limit holds:
$\lim_{n\to\infty} \P(T_n>0) = 1$.  
Hence, by Lemma \ref{lem:2nd_moment}, it is enough to prove $\var(T_n)\ll (\E T_n)^2$. 

We begin with computing $\E T_n$. 
For fixed $k$, let $\overline{B}_{0,2},\overline{B}_{1,1},\overline{B}_{1,2},\overline{B}_{2,2}$ be the sets of edges (reactions) in $E_{0,2},E_{1,1},E_{1,2},E_{2,2}$, respectively,  in which species $X_k$ appears as a non-catalyst-only species.  So, by construction, $B_k$ occurs if and only if $G_n$ contains no reaction from the sets $\overline{B}_{i,j}$:  
	\begin{align*}
	\overline{B}_{0,2}
		~&=~
			\{ 0 \leftrightarrows X_k + X_{j} \mid j \neq k,~ 1 \leq j \leq n\} \\
	\overline{B}_{1,1}
		~&=~	
			\{a X_k \leftrightarrows b X_j \mid a, b=1,2,~ j \neq k, ~ 1 \leq j \leq n \}  ~\cup ~  \{ X_k \leftrightarrows 2X_k\}
		\\
	\overline{B}_{1,2}
		~&=~		
			\{ X_k \leftrightarrows X_j + X_{\ell} \mid j \neq \ell, ~ j,\ell \neq k,~ 1 \leq j,\ell \leq n\}
			~\cup~
			\{ 2X_k \leftrightarrows X_j + X_{\ell} \mid j \neq \ell, ~1 \leq j,\ell \leq n\} \\
			& \quad \quad \quad  ~\cup~
			\{ aX_j \leftrightarrows X_K + X_{\ell} \mid j,\ell \neq k, ~1 \leq j,\ell \leq n\}
					\\
	\overline{B}_{2,2}
		~&=~	
			\{ X_k + X_i \leftrightarrows X_j + X_{\ell} \mid i,j,\ell \neq k,~ j \neq \ell,~ 1 \leq i,j,\ell \leq n \}~.
	\end{align*}

%{\color{brown}
%\sout{ 
%cannot be present in the event $B_k$. Observe that there are two vertices in $C_1$ containing $X_k$, which are the vertices $X_k$ and $2X_k$, and there are $n-1$ vertices in $C_2$ containing $X_k$. Thus there are $2n-2$ $C_1$-vertices and $(n-1)(n-2)/2$ $C_2$-vertices not containing $X_k$. With the exception of $0\leftrightarrows X_k$ and $0\leftrightarrows 2X_k$, we count the number of possible edges among vertices containing $X_k$ where $X_k$ is not a catalyst-only species, and the number of edges between a vertex containing $X_k$ and a vertex not containing $X_k$.}} 
It is then straightforward to compute the cardinalities of the sets  $\overline{B}_{i,j}$: 
\[|\overline{B}_{0,2}|=n-1, \quad |\overline{B}_{1,1}|=4n-3, \quad |\overline{B}_{1,2}|=(n-1)(3n-3), \quad |\overline{B}_{2,2}|=\frac{(n-1)^2(n-2)}{2}.\]
%\begin{itemize}
%\item In $E_{0,2}$ there are $n-1$ possible edges between $0$ and $X_k+X_i\in C_2$.
%\item In $E_{1,1}$ there is 1 possible edge between $X_k$ and $2X_k$ and $2(2n-2)$ possible edges between $X_k,2X_k$ and other vertices of $C_1$.
%\item In $E_{1,2}$ there are $(n-1)$ possible edges between $2X_k$ and $X_k+X_h\in C_2$. There are $2(n-1)(n-2)/2=(n-1)(n-2)$ possible edges between $X_k,2X_k$ and $X_i+X_j\in C_2$ where $i,j\neq k$. There are $(n-1)(2n-2)$ possible edges between $X_k+X_h\in C_2$ and $X_i,2X_i\in C_1$ where $i\neq k$. The total number of such edges is $(n-1)(3n-3)$.
%\item In $E_{2,2}$ there are $(n-1)^2(n-2)/2$ possible edges between one of $X_k+X_h\in C_2$ and $X_i+X_j \in C_2$ where $i,j\neq k$. 
%\end{itemize}
Thus, using the edge probabilities~\eqref{edge_probability} and the hypothesis $p_n < \frac{1}{n^2}$, 
we get $ \P(B_k)$ and hence $\E T_n$, as follows:
\begin{align} \label{eq:probab-b}
    \P(B_k) &~=~ (1-n^2p_n)^{n-1+4n-3}(1-np_n)^{(n-1)(3n-3)}(1-p_n)^{(n-1)^2(n-2)/2}\\
    \notag
    &~=~ (1-n^2p_n)^{5n-4}(1-np_n)^{(n-1)(3n-3)}(1-p_n)^{(n-1)^2(n-2)/2},\\
    \label{eq:expect-T}
\E T_n &~=~ \sum_{k=1}^n\P(B_k)=n\P(B_k) = n(1-n^2p_n)^{5n-4}(1-np_n)^{(n-1)(3n-3)}(1-p_n)^{(n-1)^2(n-2)/2}.
\end{align}

We recall the following, which is well known: 
\\
\noindent {\bf Fact:}
For a sequence $\lambda(n)$ and $q \geq 1$, if $\lambda(n) \ll n$, then  $(1- \frac{\lambda(n)}{n^q})^{n^q} \sim e^{-\lambda(n)}$.

%{\color{brown} 
%\sout{
%When $\lambda(n)\ll n^{1/2}$, we have $(1-\lambda(n)/n)^n\sim e^{-\lambda(n)}$. }. 
%}

To apply this fact, we use~\eqref{eq:window} to obtain %$n p_n \ll n^2 p_n \ll 
$n^3 p_n \leq  \left( \log(n) - c(n) \right) \ll n$.  We can now apply the fact (with $q=1,2,3$) to the following three factors in~\eqref{eq:expect-T}:  
$(1-n^2p_n)^{5n-4}$, 
$(1-np_n)^{(n-1)(3n-3)}$, 
and $(1-p_n)^{(n-1)^2(n-2)/2}$. One of these analyses is shown below (and the other two are similar):
\begin{small}
	\begin{align*}
	(1-np_n)^{(n-1)(3n-3)}
	~=~
	\left( \left(
	1 - \frac{n^3p_n}{n^2}
	\right)^{n^2}	
	\right)^{\frac{(n-1)(3n-3)}{n^2}}
	~\sim ~
	\left( e^{-n^3 p_n}
	\right)^{\frac{(n-1)(3n-3)}{n^2}}
	~=~
	e^{-p_n n (n-1)(3n-3)}~.
	\end{align*}
\end{small}
	
The resulting three limits combine to yield the first limit here:

%{\color{brown} 
% \sout{
%Since $p_n\leq\frac{\log(n)-c(n)}{n^3}$, we obtain $n^2p_n \ll \frac{n^{1/2}}{n}$ and thus, we get the estimate} 
%}
\[
\E T_n ~\sim~ n e^{-p_n(n^2(5n-4)+n(n-1)(3n-3)+(n-1)^2(n-2)/2)} \geq ne^{-\frac{17}{2}n^3p_n} \geq ne^{-\log(n)+\frac{17}{2}c(n)} = e^{\frac{17}{2}c(n)} \gg 1~,
\]
 and the remaining inequalities are direct computations or come from~\eqref{eq:window}.
%Similar to Proposition \ref{lemma:MM}, it suffices to show 
%\[
%\lim_{n\to\infty}\frac{\var(T)}{(\E T)^2} = 0.
%\]

Next, we compute $\P(B_k\cap B_h)$. To that end, for fixed $k, h$ (with $1 \leq k, h \leq n$ and $k\neq h$), let $\overline{A}_{0,2},\overline{A}_{1,1},\overline{A}_{1,2},\overline{A}_{2,2}$ denote the sets of edges in $E_{0,2},E_{1,1},E_{1,2},E_{2,2}$, respectively, in which species $X_k$ or $X_h$ (or both) appear as a non-catalyst-only species. 
By construction, $B_k\cap B_h$ occurs if and only if $G_n$ contains no reaction from the sets $\overline{A}_{i,j}$.  Also, $\overline{A}_{i,j}$ is the union of two sets of the form $\overline{B}_{i,j}$, one for $k$ and one for $h$.
%{\color{brown}\sout{
%that do not appear in $B_k\cap B_h$ for given $k$ and $h$. }}
Thus, the cardinalities of $\overline{A}_{i,j}$ are computed (in a straightforward way) using the inclusion-exclusion principle: 
%A straightforward {\color{violet} (inclusion-exclustion)} computation gives:
	\begin{align*}
	|\overline{A}_{0,2}| ~&=~ 2 | \overline{B}_{0,2}| -1~, \hspace{.61in}
	%\\
	|\overline{A}_{1,1}| ~=~2 | \overline{B}_{1,1}| -4~, \quad \\
	|\overline{A}_{1,2}| ~&=~2 | \overline{B}_{1,2}| -  4(n-2)~, %\\ 
	\quad 
	 |\overline{A}_{2,2}|  ~=~ 2 | \overline{B}_{2,2}| -  (n-2)(3n-7)/2 ~.	
%	|\overline{A}_{0,2}| ~&=~2(n-1)-1~, \quad 
%	\\
%	|\overline{A}_{1,1}| ~&=~2(4n-3)-4~, \quad \\
%	|\overline{A}_{1,2}| ~&=~2(n-1)(3n-3)-4(n-2)~, \\ 
%	 |\overline{A}_{2,2}|  ~&=~ 2\frac{(n-1)^2(n-2)}{2}-(n-2)(n-3)~.
	\end{align*}
%\begin{itemize}
%    \item In $E_{0,2}$ there are $n-1$ possible edges which must not appear in each of the events $B_k$ and $B_h$. Since the edge $0 \leftrightarrows X_k+X_h$ must not appear in both, the number of possible edges in $E_{0,2}$ that must not be present in the event $B_k\cap B_h$ is $2(n-1)-1$.
%    \item In $E_{1,1}$ there are $1+2(2n-2)=4n-3$ possible edges which must not appear in each of the events $B_k$ and $B_h$. There are 4 edges between either $B_k$ or $2B_k$ and $B_h$ or $2B_h$ which must not appear in both, so the number of possible edges in $E_{1,1}$ that must not be present in the event $B_k\cap B_h$ is $2(4n-3)-4$.
%    \item In $E_{1,2}$ there are $(n-1)(3n-3)$ possible edges which must not appear in each of the events $B_k$ and $B_h$. There are $2(n-2)$ possible edges between $X_k$ or $2X_k$ and $X_h+X_i\in C_2$ and there are $2(n-2)$ possible edges between $X_h$ or $2X_h$ and $X_k+X_i\in C_2$ which must not appear in both events. So the number of possible edges in $E_{1,2}$ that must not be present in the event $B_k\cap B_q$ is $2(n-1)(3n-3)-4(n-2)$.
%    \item In $E_{2,2}$ there are $(n-1)^2(n-2)/2$ possible edges which must not appear in each of the events $B_k$ and $B_h$. There are $(n-2)(n-3)$ possible edges between $X_k+X_i\in C_2$ to $X_p+X_j\in C_2$ which must not appear in both events. Thus the number of possible edges in $E_{2,2}$ that must not be present in the event $B_k\cap B_h$ is $2(n-1)^2(n-2)/2-(n-2)(n-3)$.
%\end{itemize}
The ``exclusion'' terms above yield:
\begin{align} \label{eq:prob-intersect-B}
    \P(B_k\cap B_h) &~=~ \P(B_k)^2 (1-n^2p_n)^{-5} (1-np_n)^{-4(n-2)}(1-p_n)^{-(n-2)(3n-7)/2}.
\end{align}
Using~\eqref{eq:prob-intersect-B} and other expressions found above, we compute $\var(T_n)$:
\begin{small}
\begin{align*}
    \var(T_n) &~=~ \E(T_n^2)-(\E T_n)^2 \\
    &~=~ \sum_{k=1}^n\P(B_k) + \sum_{h\neq k}\P(B_k\cap B_h) - (\E T_n)^2 \\
    &~=~ n\P(B_k) + n(n-1)\P(B_k)^2(1-n^2p_n)^{-5} (1-np_n)^{-4(n-2)}(1-p_n)^{-(n-2)(3n-7)/2} - n^2\P(B_k)^2 \\
    &~=~ n\P(B_k)-n\P(B_k)^2g(n)+n^2\P(B_k)^2(g(n)-1)~,
\end{align*}
\end{small}
where $g(n):=(1-n^2p_n)^{-5} (1-np_n)^{-4(n-2)}(1-p_n)^{-(n-2)(3n-7)/2}$.  We claim that $\lim_{n\to\infty}g(n)=1$.  
%\[
%1 ~\leq~ g(n) ~\leq~ e^{-p_n(5n^2+4n(n-2)+(n-2)(n-3))}.
%\]
In fact, since $n^2p_n \ll 1$, we have  (for $n$ sufficiently large) the inequalities 
$1 \leq (1-n^2p_n)^{-5} $ 
and 
$\log(1-n^2p_n) \geq -2n^2p_n$, the second of which further implies 
$ (1-n^2p_n)^{-5} \leq e^{10n^2p_n}$. 
Applying similar inequalities for the remaining two factors of $g(x)$, we obtain:
	\begin{align} \label{eq:g(n)}
	1 ~\leq~ g(n) ~\leq~ e^{p_n \left(10n^2+8(n-2)n+(n-2)(3n-7) \right)} \quad \quad  \mathrm{for}~n~\mathrm{sufficently~large}.
	\end{align}

By~\eqref{eq:window}, the exponent appearing in~\eqref{eq:g(n)}
limits to 0,
 %$p_n(10n^2+8(n-2)n+2(n-2)(n-3)) \to 0$} 
 as $n\to\infty$, so indeed $g(n)\to 1$. % as $n\to\infty$. 
 Hence,
\begin{align*}
    \frac{\var(T_n)}{(\E T_n)^2} 
%    & = \frac{ n\P(B_k)-n\P(B_k)^2g(n)+n^2\P(B_k)^2(g(n)-1)}{n^2\P(B_k)^2}\\
    &~=~\frac{1}{n\P(B_k)}-\frac{g(n)}{n}+g(n)-1  
    ~=~\frac{1}{\E T_n}-\frac{g(n)}{n}+g(n)-1
    ~ \to ~ 0 + 0 + 1 - 1  ~=~0~, 
%    \\
%    &=\frac{1}{\E T_n}-\frac{g(n)}{n}+g(n)-1
\end{align*}
as $n \to \infty$, 
where we also use $\E T_n\gg 1$. % and $\lim_{n\to\infty}g(n)=1$, we have 
Thus, 
%$\lim_{n\to\infty}\frac{\var(T_n)}{(\E T_n)^2}=0$, 
$\var(T_n)\ll (\E T_n)^2$, 
which completes the proof.
%as desired.
\end{proof}

\begin{remark}[Decoupling in window of dense regime] \label{remark:decouple} 
The proof of Proposition \ref{prop:ACR} %(particularly from the definition of the event $B_k$) 
shows that, if $\frac{1}{n^3}\ll p_n \leq \frac{\frac{2}{17}\log(n)-c(n)}{n^3}$, then w.h.p.\ a random network $G_n$ contains a subnetwork of the form $\{0 \leftrightarrows X_k, ~ 0 \leftrightarrows 2X_k\}$ for some species $X_k$ that is a catalyst-only species in all other reactions. 
This highlights the fact that unconditional ACR 
arises because $G_n$ is a union of two ``almost decoupled''  subnetworks, one with ACR and the other with multistationarity (by Proposition~\ref{prop:mss}) w.h.p.
%mainly as a result of being a union of two subnetworks, one with ACR and the other with multistationarity, in an ``almost decoupled'' way. 
\end{remark}

\section{Discussion} \label{sec:discussion}
%\begin{itemize}
%	\item Hint at future paper(s)
%	\item Medium-number of species -- how to check multistationarity?  (Challenge.  Maybe want database of small multistationary networks, to look for inside the original network.  Connect with ``co-existence'' paper.)
%\end{itemize}

We have shown that it is highly atypical for multistationarity and ACR to coexist
in certain random reaction networks. In particular, for the type-homogeneous stochastic block model, the window for co-existence is relatively small: It corresponds to when the expected number of edges is approximately between $n$ and $\frac{2}{17}n\log(n)$, where $n$ is the number of species. 
%{\cre (there is no $c_n$ here. Is that okay?)}
This window does not even exist unless $n$ is quite large  (Remark~\ref{remark:window}).  Moreover, 
when this window exists, the resulting 
random networks % in this window (when it exists) 
exhibit multistationarity and ACR simply as a result of 
nearly decoupling into 
%being nearly a mutually exclusive union of 
two subnetworks, one with ACR and the other with multistationarity (Remark \ref{remark:decouple}).

These results suggest that reaction networks that combine multistationarity and ACR in a nontrivial way require specialized architecture, and the properties do not occur together coincidentally. 
Of course, real biochemical networks are far from random and exist only %in nature 
when they offer a selective advantage to the organism in its environment. 
It is a reasonable speculation that combining the two seemingly opposite properties may be favorable. 
A biochemical network may require robustness in its internal operation while maintaining flexibility as a signal-response mechanism. 
%{\color{violet} \sout{Certain species concentrations maintain invariability even when the system is open and interacting, while other species concentrations show flexibility in their response that distinguishes between distinct input signals.}}
Said differently,  such a network may operate through an essential combination of ACR with multistability. 

These ideas raise a natural question: Which special structures, even if statistically rare, can produce ACR and multistability in networks of biochemically reasonable size and complexity?
In future work, we will report on such mechanisms and their underlying principles. % that generate coexistence of the two traits. 
Interestingly, we find families of biochemical networks with ACR and multistationarity that %moreover 
employ ubiquitous designs such as enzyme-catalyzed reactions, lock-and-key mechanisms for enzyme binding, and redundancy through parallel pathways. 
%{\color{brown} Add transition here?  it went from ``architecture'' to a new topic of ``flexible vs.\ robust'', and later in the paragraph returns to architecture/structure.}
%Multistationarity provides a basis for a flexible response while ACR undergirds robustness. 
%The two seemingly contradictory features can coexist along independent coordinate axes. 
%Combining a flexible, adaptable response with robustness that preserves essential functions may in fact be evolutionarily favored. 
%It is therefore conceivable that natural selection discovers the special structures that endow a biochemical network with both multiple stable steady states as well as absolute concentration robustness despite their mathematical rarity. 
%In future work, we present and study in depth interesting families of biochemical networks that combine the two traits. Moreover, the networks use ubiquitous biochemical motifs such as enzyme-catalyzed reactions, lock-and-key mechanisms for enzyme binding, and redundancy through parallel pathways. 
% structures or careful constructions. Such a class of networks is the topic of future work; interestingly, these networks display a ``lock and key'' mechanism that is relevant for many biological processes.
% NOTE: induced-fit model is now regarded as a better theory for enzymatic reactions than the original "lock and key" mechanism proposed in 1890 by Emil Fischer.

Returning to the current work, we gave asymptotic results on multistationarity when $n$ (the number of species) is large.  We are also interested in multistationarity when $n$ is of medium size (say, $n=10$ to $30$).  We would like to investigate, by generating random such networks (at various edge-probabilities), what fraction are multistationary. Although checking multistationarity is generally difficult, an approach used here --
%we used in this work for establishing multistationarity -- 
namely, finding a small multistationary motif (ours had only $3$ species) and then lifting it -- can be applied.  
%Specifically, %the algorithm 
%would first check the original network and attempt to find any small multistationary subnetwork. Next, the algorithm would check if there is a lifting component that can lifts multistationarity to the whole network (see CITE for a survey of possible lifting conditions). 
For performing this task, note that certain classes of small multistationary networks have been established~\cite{atoms_multistationarity, joshi2017small, tang2021multistability}, as have various criteria for lifting multistationarity (surveyed in~\cite{splitting-banaji}).%, mss-review}.  

Going forward, it would be interesting to discover more small multistationary motifs.  Are there more multistationary networks with only $3$ species that are well suited for lifting to larger networks?  Establishing such networks might aid in analyzing the prevalence of multistationarity -- with or without ACR -- in random reaction networks generated by stochastic block models besides the type-homogeneous one we focused on here. 

A final promising direction is to study the prevalence and thresholds of other reaction-network properties. In particular, properties that can be lifted from small networks to larger ones -- such as periodic orbits \cite{splitting-banaji, banaji-boros, erban-kang, tang-wang} 
-- can also be analyzed in our random-network framework. % we established in the current paper.
Do periodic orbits co-exist with ACR in random networks?  If so, then, as is the case for multistationarity and ACR, the window of co-existence is likely very small.

%it is important to extend the database of known small multistationary networks, for instance, by investigating multistationarity in networks with three species. In particular, we could extend the motif we use in the current paper to a more general class of networks, and find other classes of 3-species networks exhibiting multistationarity.

%{\color{red} Another future direction: Prevalence of oscillations.}

%{\color{blue} Temporary: suggested editors -- Rubin, Gedeon, Qian; suggested referees: Banaji, Cappelletti, Craciun, Boros, Jinsu Kim.  }
%{\color{brown} Remove this comment after a cover letter is written, with this info.}

%*******************************************************************
%Acknowledgements
%*******************************************************************
\subsection*{Acknowledgements} {\small
This project began at an AIM workshop on ``Limits and control of stochastic reaction networks'' held online in July 2021.  
AS was supported by the NSF (DMS-1752672).  
BJ was supported by the NSF (DMS-2051498). 
The authors thank Elisenda Feliu for helpful discussions and David F. Anderson for comments on an earlier draft.
}

\bibliographystyle{plain}
	\bibliography{bib}

\begin{thebibliography}{10}

\bibitem{alon2016probabilistic}
Noga Alon and Joel~H Spencer.
\newblock {\em The probabilistic method}.
\newblock John Wiley \& Sons, 2016.

\bibitem{anderson2017finite}
David~F Anderson, Daniele Cappelletti, and Thomas~G Kurtz.
\newblock Finite time distributions of stochastically modeled chemical systems
  with absolute concentration robustness.
\newblock {\em SIAM Journal on Applied Dynamical Systems}, 16(3):1309--1339,
  2017.

\bibitem{AC:non-mass}
David.~F. Anderson and Simon.~L. Cotter.
\newblock Product-form stationary distributions for deficiency zero networks
  with non-mass action kinetics.
\newblock {\em Bulletin of Mathematical Biology}, 78(12), 2016.

\bibitem{ACK:product}
David~F. Anderson, Gheorghe Craciun, and Thomas~G. Kurtz.
\newblock Product-form stationary distributions for deficiency zero chemical
  reaction networks.
\newblock {\em Bulletin of Mathematical Biology}, 72(8), 2010.

\bibitem{AN:non-mass}
David~F. Anderson and Tung~D. Nguyen.
\newblock Results on stochastic reaction networks with non-mass action
  kinetics.
\newblock {\em Mathematical Biosciences and Engineering}, 16(4):2118--2140,
  2019.

\bibitem{prevalence_block}
David~F Anderson and Tung~D Nguyen.
\newblock Deficiency zero for random reaction networks under a stochastic block
  model framework.
\newblock {\em Journal of Mathematical Chemistry}, 59(9):2063--2097, 2021.

\bibitem{prevalence}
David~F. Anderson and Tung~D. Nguyen.
\newblock Prevalence of deficiency zero reaction networks in an {E}rdos-{R}enyi
  framework.
\newblock {\em Journal of Applied Probability}, 59(2):384--398, 2022.

\bibitem{splitting-banaji}
Murad Banaji.
\newblock Splitting reactions preserves nondegenerate behaviours in chemical
  reaction networks.
\newblock {\em Preprint, {\tt arXiv:2201.13105}}, 2022.

\bibitem{banaji-boros}
Murad Banaji and Bal{\'a}zs Boros.
\newblock The smallest bimolecular mass action reaction networks admitting
  {A}ndronov-{H}opf bifurcation.
\newblock {\em Preprint, {\tt arXiv:2207.04971}}, 2022.

\bibitem{lifting_mss}
Murad Banaji and Casian Pantea.
\newblock The inheritance of nondegenerate multistationarity in chemical
  reaction networks.
\newblock {\em SIAM Journal on Applied Mathematics}, 78(2):1105--1130, 2018.

\bibitem{threshold}
B\'ela Bollob\'as and Andrew Thomason.
\newblock Threshold functions.
\newblock {\em Combinatorica}, 7:35--38, 1987.

\bibitem{boros2019existence}
Bal{\'a}zs Boros.
\newblock Existence of positive steady states for weakly reversible mass-action
  systems.
\newblock {\em SIAM Journal on Mathematical Analysis}, 51(1):435--449, 2019.

\bibitem{briat2016antithetic}
Corentin Briat, Ankit Gupta, and Mustafa Khammash.
\newblock Antithetic integral feedback ensures robust perfect adaptation in
  noisy biomolecular networks.
\newblock {\em Cell systems}, 2(1):15--26, 2016.

\bibitem{Deng}
Jian Deng, Martin Feinberg, Chris Jones, and Adrian Nachman.
\newblock On the steady states of weakly reversible chemical reaction networks.
\newblock Preprint, {\tt arXiv:1111.2386}.

\bibitem{erban-kang}
Radek Erban and Hye-Won Kang.
\newblock Chemical systems with limit cycles.
\newblock {\em Preprint, {\tt arXiv:2211.05755}}, 2022.

\bibitem{ER:giantcomponent}
Paul Erd{\H o}s and Alfr\'ed R{\'e}nyi.
\newblock On the evolution of random graphs.
\newblock In {\em Publication of the Mathematical Institute of the Hungarian
  Academy of Sciences}, pages 17--61, 1960.

\bibitem{F1}
Martin Feinberg.
\newblock Complex balancing in general kinetic systems.
\newblock {\em Archive for Rational Mechanics and Analysis}, 49:187--194, 1972.

\bibitem{feinberg2019foundations}
Martin Feinberg.
\newblock {\em Foundations of chemical reaction network theory}.
\newblock Springer, 2019.

\bibitem{holland1983stochastic}
Paul~W Holland, Kathryn~Blackmond Laskey, and Samuel Leinhardt.
\newblock Stochastic blockmodels: {F}irst steps.
\newblock {\em Social networks}, 5(2):109--137, 1983.

\bibitem{H}
Fritz Horn.
\newblock Necessary and sufficient conditions for complex balancing in chemical
  kinetics.
\newblock {\em Archive for Rational Mechanics and Analysis}, 49:172--186, 1972.

\bibitem{H-J1}
Fritz Horn and Roy Jackson.
\newblock General mass action kinetics.
\newblock {\em Archive for Rational Mechanics and Analysis}, 47:187--194, 1972.

\bibitem{joshi2022motifs}
Badal Joshi and Gheorghe Craciun.
\newblock Reaction network motifs for static and dynamic absolute concentration
  robustness.
\newblock {\em to appear in SIAM Journal on Applied Dynamical Systems}.

\bibitem{foundation_ACR}
Badal Joshi and Gheorghe Craciun.
\newblock Foundations of static and dynamic absolute concentration robustness.
\newblock {\em Journal of Mathematical Biology}, 85(53), 2022.

\bibitem{atoms_multistationarity}
Badal Joshi and Anne Shiu.
\newblock Atoms of multistationarity in chemical reaction networks.
\newblock {\em Journal of Mathematical Chemistry}, 51:153--178, 2013.

\bibitem{mss-review}
Badal Joshi and Anne Shiu.
\newblock {A} survey of methods for deciding whether a reaction network is
  multistationary.
\newblock {\em Math. Model. Nat. Phenom., special issue on ``Chemical
  dynamics''}, 10(5):47--67, 2015.

\bibitem{joshi2017small}
Badal Joshi and Anne Shiu.
\newblock Which small reaction networks are multistationary?
\newblock {\em SIAM Journal on Applied Dynamical Systems}, 16(2):802--833,
  2017.

\bibitem{kim2020absolutely}
Jinsu Kim and German Enciso.
\newblock Absolutely robust controllers for chemical reaction networks.
\newblock {\em Journal of the Royal Society Interface}, 17(166):20200031, 2020.

\bibitem{MST}
Nicolette Meshkat, Anne Shiu, and Angelica Torres.
\newblock Absolute concentration robustness in networks with low-dimensional
  stoichiometric subspace.
\newblock {\em Vietnam Journal of Mathematics}, 50:623--651, 2022.

\bibitem{ACR}
Guy Shinar and Martin Feinberg.
\newblock Structural sources of robustness in biochemical reaction networks.
\newblock {\em Science}, 327(5971):1389--1391, 2010.

\bibitem{tang-wang}
Xiaoxian Tang and Kaizhang Wang.
\newblock Hopf bifurcations of reaction networks with zero-one stoichiometric
  coefficients.
\newblock {\em Preprint, {\tt arXiv:2208.04196}}, 2022.

\bibitem{tang2021multistability}
Xiaoxian Tang and Hao Xu.
\newblock Multistability of small reaction networks.
\newblock {\em SIAM Journal on Applied Dynamical Systems}, 20(2):608--635,
  2021.

\bibitem{tyson-albert}
John~J Tyson, Reka Albert, Albert Goldbeter, Peter Ruoff, and Jill Sible.
\newblock Biological switches and clocks.
\newblock {\em J.\ R.\ Soc.\ Interface}, 5:S1--S8, 2008.

\end{thebibliography}

\end{document}